\documentclass[10.5pt,letterpaper]{article}
 \usepackage[hscale=0.75,vscale=0.8]{geometry}

\usepackage[sort&compress]{natbib}
\renewcommand{\bibname}{References}
\renewcommand{\bibsection}{\subsubsection*{\bibname}}

\usepackage{url}            
\usepackage{booktabs}       
\usepackage{nicefrac}       

\usepackage{amsmath,amsfonts,amssymb,amsthm}
\usepackage{relsize}
\usepackage{algorithm,algorithmic}

\usepackage{multirow}
\usepackage{cellspace}
\usepackage{setspace}
\usepackage{bm}
\usepackage{bbm}
\usepackage{wrapfig}

\usepackage{graphicx}

\usepackage{subcaption}

\usepackage{enumitem}

\usepackage{hyperref}
\hypersetup{colorlinks,
            linkcolor=blue,
            citecolor=blue,
            urlcolor=magenta,
            linktocpage,
            plainpages=false}

\DeclareMathOperator*{\argmin}{arg\,min}
\DeclareMathOperator*{\argmax}{arg\,max}

\def\R{{\mathbb{R}}}

\newcommand{\norm}[1]{{ \left\lVert#1\right\rVert }}

\newcommand{\E}{\mathbb{E}}

  \usepackage{nth}
  \usepackage{intcalc}

\newcommand{\ignore}[1]{}

\def\bold0{\mathbf{0}}

\newtheorem*{rep@theorem}{\rep@title}
\newcommand{\newreptheorem}[2]{%
\newenvironment{rep#1}[1]{%
 \def\rep@title{#2 \ref{##1}}%
 \begin{rep@theorem}}%
 {\end{rep@theorem}}}
\newreptheorem{lemma}{Lemma'}
\newreptheorem{definition}{Definition'}
\newreptheorem{proposition}{Proposition'}
\newreptheorem{theorem}{Theorem'}

\newtheorem{theorem}{Theorem}

\renewcommand{\text}[1]{{\mathrm{#1}}}
\renewcommand{\text}[1]{{\textnormal{#1}}}

\DeclareMathOperator{\ranges}{ranges}
\DeclareMathOperator{\range}{range}

\newtheorem{lemma}[theorem]{Lemma}
\newtheorem{corollary}[theorem]{Corollary}

\newtheorem*{theorem*}{Theorem}
\newtheorem*{lemma*}{Lemma}
\newtheorem*{corollary*}{Corollary}
\newtheorem*{claim*}{Claim}
\newtheorem*{fact*}{Fact}
\newtheorem*{remark*}{Remark}

\theoremstyle{definition}
\newtheorem{definition}[theorem]{Definition}
\theoremstyle{definition}

\def\1{\mathbf{1}}

\def\1{\mathbf{1}}

\title{Importance Sampling via Local Sensitivity}
\author{Anant Raj  \\
  MPI for Intelligent Systems,\\
  T\"ubingen, Germany.\\
  \texttt{anant.raj@tuebingen.mpg.de} 
\and Cameron Musco \thanks{Most of this work was done when Cameron Musco was at Microsoft Research.}\\
  UMass Amherst, \\
  Amherst, USA. \\
  \texttt{cmusco@cs.umass.edu } \\
  \and Lester Mackey \\
  Microsoft Research, \\
  New England, USA. \\
  \texttt{lmackey@stanford.edu} 
}

\begin{document}
\maketitle
\begin{abstract} 
Given a loss function $F:\mathcal{X} \rightarrow \R^+$ that can be written as the sum of losses over a large set of inputs $a_1,\ldots, a_n$,  it is often desirable to approximate $F$ by subsampling the input points. Strong theoretical guarantees require taking into account the importance of each point, measured by how much its individual loss contributes to $F(x)$. Maximizing this importance over all $x \in \mathcal{X}$ yields the \emph{sensitivity score} of $a_i$. Sampling with probabilities proportional to these scores gives strong guarantees, allowing one to approximately minimize of $F$ using just the subsampled points.

Unfortunately, sensitivity sampling is difficult to apply since (1) it is unclear how to efficiently compute the sensitivity scores and (2) the sample size required is often impractically large. To overcome both obstacles we introduce \emph{local sensitivity}, which measures data point importance in a ball around some center $x_0$. We show that the local sensitivity can be efficiently estimated using the \emph{leverage scores} of a quadratic approximation to $F$ and that the sample size required to approximate $F$ around $x_0$ can be bounded. We propose employing local sensitivity sampling in an iterative optimization method  and analyze its convergence when $F$ is smooth and convex.

\end{abstract}

\section{Introduction}

In this work we consider finite sum minimization problems of the following form.
\begin{definition}[Finite Sum Problem]\label{def:fs} Given data points $a_1,\ldots, a_n \in \R^d$, nonnegative functions $f_1,\ldots, f_n: \R \rightarrow \R^+$, and a nonnegative function $\gamma: \R^d \rightarrow \R^+$, minimize over $x \in \mathcal{X} \subseteq \R^d$
\begin{align}\label{eq:fs}
    F(x) :=\frac{1}{n} \sum_{i=1}^n f_i(a_i^T  x) + \gamma(x).
\end{align}
\end{definition}
Definition \ref{def:fs} captures a number of important problems, including penalized empirical risk minimization (ERM) for linear regression, generalized linear models, and support vector machines.
When $n$ is large, minimizing $F(x)$ can be expensive. In some cases, for example, it may be impossible to load the full dataset $a_1,\ldots,a_n$ into memory.

\subsection{Function Approximation via Data Subsampling}
To reduce the burden of solving a finite sum problem, one commonly minimizes an approximation to $F$ formed by independently subsampling data points $a_i$ (and hence summands $f_i(a_i^T  x)$) with some fixed probability weights.
More formally:

\begin{definition}[Subsampled Finite Sum Problem]\label{def:subSampled}
Consider the setting of Definition \ref{def:fs}. Given a target sample size $m$ and a probability distribution $P = \{p_1,\ldots,p_n\}$ over $[n] \triangleq \{1,\dots,n\}$, select $i_1,...,i_m$ i.i.d.\ from $P$ and minimize over $x \in \mathcal{X} \subseteq \R^d$
 
\begin{align}
    F^{(P,m)}(x) := \frac{1}{mn} \sum_{j=1}^m \frac{ f_{i_j}(a_{i_j}^T x)}{p_{i_j}} + \gamma(x).
\end{align}
  
\end{definition}
We can see that for any $x$, $\E[F^{(P,m)}(x)] = F(x)$. If the sampled function concentrates well around $F(x)$, then it can serve effectively as a surrogate for minimizing $F$. 
Most commonly, $P$ is set to the uniform distribution. Unfortunately,
if $F(x)$ is dominated by the values of a relatively few large $f_i(a_i^T x)$, unless $m$ is very large, uniform subsampling will miss these important  data points and $F^{(P,m)}(x)$ will often underestimate $F(x)$. 
This can happen, for example, when $a_1,...,a_n$ fall into clusters of non-uniform size. Data points in smaller clusters are important in selecting an optimal $x$ but are often underrepresented in a uniform sample.

\subsection{Importance Sampling via Sensitivity}
A remedy to the weakness of uniform subsampling is to apply importance sampling: preferentially sample  the functions $f_i(a_i^T x)$ that  contribute most significantly to $F(x)$. If, for example, we set $p_i \propto \frac{f_i(a_i^T x)}{\sum_{i=1}^n f_i(a_i^T x) + \gamma(x)}$  for each $i \in [n]$, then a standard concentration argument would imply that $(1-\epsilon) F(x) \le F^{(P,m)}(x) \le (1+\epsilon) F(x)$ with probability  at least $\1-\delta$ if $m = \Theta \left (\frac{\log(1/\delta)}{\epsilon^2}\right)$. However, typically the relative the importance of each point, $\frac{f_i(a_i^T x)}{\sum_{i=1}^n f_i(a_i^T x)+\gamma(x)}$, will depend on the choice of $x$. This motivates the definition of  \emph{sensitivity}  \citep{langberg2010universal}.

\begin{definition}[Sensitivity]\label{def:sensitivity}
For $a_1,\ldots, a_n \in \R^d$, the \emph{sensitivity} of point $a_i$ with respect to a finite sum function $F$ (Definition \ref{def:fs}) with domain $\mathcal{X} \subseteq \R^d$ is
\begin{align*}
\sigma_{F,\mathcal{X}}(a_i) = \sup_{x \in \mathcal{X}} \frac{f_i(a_i^Tx)}{\sum_{j=1}^n f_j(a_j^T x) + n\gamma(x)}.
\end{align*}

The \textit{total sensitivity} is defined as $\mathcal{G}_{F,\mathcal{X}} = \sum_{i = 1}^n \sigma_{F,\mathcal{X}}(a_i)$.
\end{definition}

A standard concentration argument yields the following approximation guarantee for sensitivity sampling.
\begin{lemma}\label{intro:lem}
Consider the setting of Definition \ref{def:fs}. For all $i \in [n]$, let $s_i \geq \sigma_{F,\mathcal{X}}(a_i)$, $S = \sum_{i =1}^n s_i$, and $P = \left \{\frac{s_1}{S},\ldots,\frac{s_n}{S} \right \}$.
There is a fixed constant $c$ such that, for any $\epsilon,\delta \in (0,1)$, any fixed $x \in \mathcal{X}$, and $m \ge \frac{c \cdot S \log(2/\delta)}{\epsilon^2}$,
\begin{align*}
   (1-\epsilon) F(x) \le F^{(P,m)}(x) \le (1+\epsilon) F(x)
\end{align*}
 with probability $\ge 1-\delta$.
\end{lemma}
That is, subsampling data points by their sensitivities approximately preserves the value of $F$ \emph{for any fixed $x \in \mathcal{X}$} with high probability. It can thus be argued that  $F$ can be approximately minimized by minimizing the sampled function $F^{(P,m)}$. We first define:
\begin{definition}[Range Space]\label{def:rs}
A range space is a pair $\mathcal{R} = (\mathcal{F},\ranges)$, where $\mathcal{F}$ is a set and $\ranges$ is a set of subsets of $\mathcal{F}$. The VC dimension $\Delta(\mathcal{R})$ is the size of the largest $G \subseteq \mathcal{F}$ such that $G$ is shattered by $\ranges$: i.e., $|\{G \cap R | R \in \ranges \}| = 2^{|G|}$.

Let $\mathcal{F}$ be a finite set of functions mapping $\R^d \rightarrow \R^+$. For every $x \in \R^d$ and $r \in \R^+$, let $\range_{\mathcal{F}}(x,r) = \{ f \in \mathcal{F} | f(x) \ge r \}$ and $\ranges(\mathcal{F}) = \{\range_{\mathcal{F}}(x,r) | x \in \R^d, r \in \R^+\}$. We say $R_\mathcal{F} = (\mathcal{F},\ranges(\mathcal{F}))$ is the range space induced by $\mathcal{F}$.
\end{definition}

With the notion of range space in place, we can recall the following general approximation theorem.

\begin{theorem}[Theorem 9 \citep{munteanu2018coresets}]\label{thm:sensitivitySampling}
Consider the setting of Definition \ref{def:fs}. For all $i \in [n]$, let $s_i \geq \sigma_{F,\mathcal{X}}(a_i)$, $S = \sum_{i =1}^n s_i$, and  $P = \left \{\frac{s_1}{S},\ldots,\frac{s_n}{S} \right \}$. 
For some finite $c$ and all $\epsilon,\delta \in (0,1/2)$, if
\begin{align*}
m \ge c \cdot \frac{S}{\epsilon^2} \left( \Delta \log S + \log \left(\frac{1}{\delta} \right) \right),
\end{align*}
then,  with probability  at least $1 - \delta$, 
\begin{align*}
   (1-\epsilon) F(x) \le F^{(P,m)}(x) \le (1+\epsilon) F(x), \forall x \in \mathcal{X}
\end{align*}
Here, $\Delta$ is an upper bound on the VC-dimension $\Delta(\mathcal{R}_\mathcal{F})$ where $\mathcal{F}$ is the set  $\left \{\frac{f_1(a_1^T x)}{mn\cdot p_1}, \ldots, \frac{f_n(a_n^T x)}{mn\cdot p_n} \right \}$  .
\end{theorem}

\citet{munteanu2018coresets} show that $\Delta = d+1$ suffices for logistic regression where $d$ is the dimension of the input points. If all $f_i$ are from the class of invertible functions, then a similar bound on $\Delta$ can be expected.

\subsubsection{Barriers  to the Sensitivity Sampling in Practice}
Theorem \ref{thm:sensitivitySampling} is quite powerful: it can be used to achieve sensitivity-sampling-based approximation algorithms with provable guarantees for a wide range of problems \citep{feldman2011unified, lucic2015strong,huggins2016coresets,munteanu2018coresets}.
However, there are two major barriers that have hindered more widespread practical adoption of sensitivity  sampling:

\begin{enumerate}

     \item \textbf{Computability:} It is difficult to compute or even approximate the sensitivity $\sigma_{F,\mathcal{X}}(a_i)$ since it is not clear how to take the supremum over all $x \in \mathcal{X}$ in the expression of Definition \ref{def:sensitivity}. Closed form expressions for the sensitivity are known only  in a few special cases, such as least squares regression (where the sensitivity is closely  related to the well-studied \emph{statistical leverage scores}).

  \item \textbf{Pessimistic Bounds:} The sensitivity score is a very `worst  case' importance metric, since it considers the supremum of $\frac{f_i(a_i^T x)}{\sum_{j=1}^n f_j(a_j^T x) + n \gamma(x)}$ over all $x \in \mathcal{X}$, including, e.g., $x$ that may be very far from the true minimizer of $F$. In many cases, it is possible to construct, for each $a_i$, some worst case $x$ that forces this ratio to be high. Thus, all sensitivities are large and the total sensitivity $\mathcal{G}_{F,\mathcal{X}}$ is large. The sample complexities in Lemma \ref{intro:lem} and Theorem \ref{thm:sensitivitySampling} depend on $S \ge \mathcal{G}_{F,\mathcal{X}}$ and so will be too large to be useful in practice. See Figure \ref{fig:largeSensitivity} for a simple example of when this issue can arise.
\end{enumerate}
\subsection{Our Approach: Local Sensitivity}

We propose to overcome the above barriers via a simple idea: \emph{local sensitivity.} Instead of sampling with the sensitivity over the full domain $\mathcal{X}$ as in Definition \ref{def:sensitivity}, we consider the sensitivity  over a small ball. Specifically, for some radius $r$ and center $y$ we let $B(r,y) = \{x \in \R^d:\ \norm{x-y} < r\}$ and consider $\sigma_{F,\mathcal{X} \cap B(r,y)}(a_i)$.
Sampling by this local sensitivity  will give us a function $F^{(P,m)}$ that \emph{approximates $F$ well on the entire ball $B(r,y)$}. Thus, we can approximately minimize $F$ on this ball. We can approximately minimize $F$ globally via an iterative scheme: at each step we set $x_i$ to the approximate optimum of $F$ over the ball $B(r_i,x_{i-1})$ (computed via local sensitivity sampling). This approach has two major advantages:

\noindent 1. We can often locally approximate each $F$ by  a simple function, for which we can compute the local sensitivities in closed form. This will yield an approximation to the true local sensitivities. Specifically, we will consider a local quadratic approximation to $F$, whose sensitivities are given by the \emph{leverage scores} of an appropriate matrix.

\noindent 2. By definition, the local sensitivity $\sigma_{F,\mathcal{X} \cap B(r,y)}$ is \emph{always} upper bounded by the global sensitivity $\sigma_{F,\mathcal{X}}$, and typically the sum of local sensitivities will be much smaller than the total sensitivity $\mathcal{G}_{F,\mathcal{X}}$. This allows us to take fewer samples to approximately  minimize $F$ locally over $B(r,y)$.
\begin{figure}
\centering
\includegraphics[width=.52\textwidth]{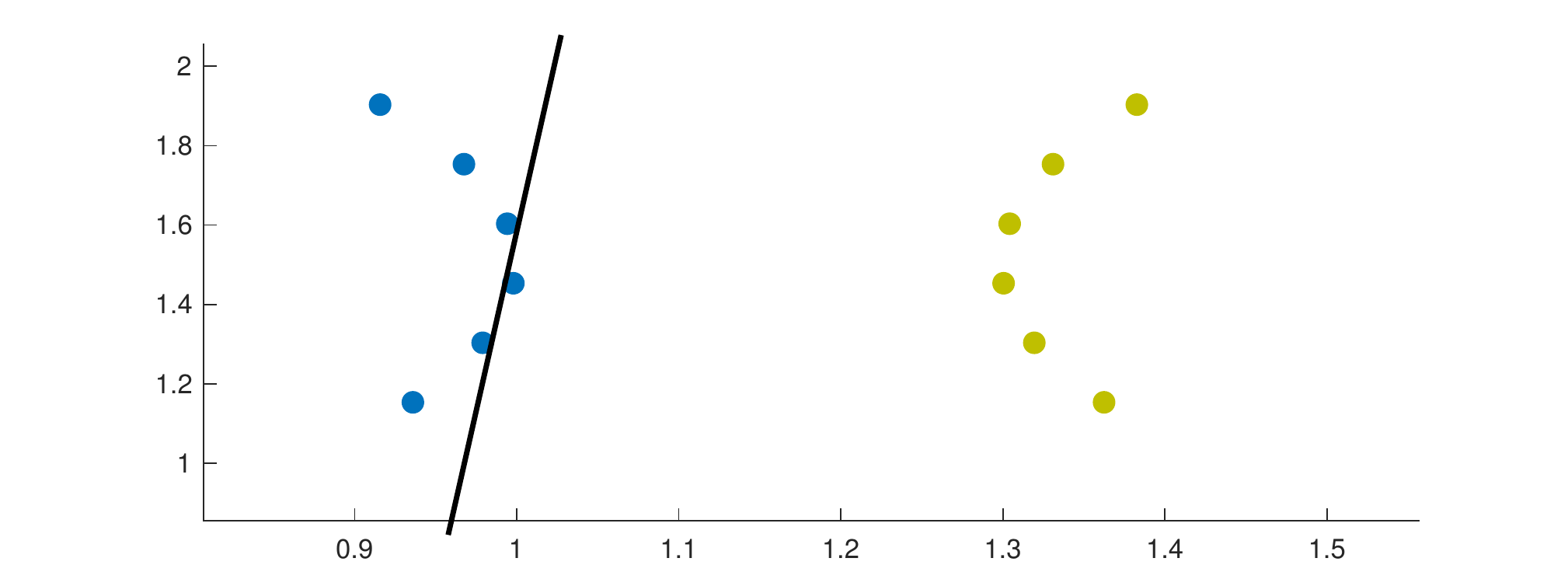}
\caption{Consider a classification problem with two classes $A_1,A_2$, shown in blue and green. Let $f_i(a_i^T x)$ be any loss function with $f_i(a_i^T x) = 0$ if $a_i$ is correctly classified by the hyperplane defined by $x$. Since for each $a_i$, there is some $x$ (e.g., corresponding to the black line shown) that misclassifies \emph{ only $a_i$}, we have $\sigma_{\mathcal{F},\R^d}(a_i) = 1$ for all $a_i$. Thus, the total sensitivity is $\mathcal{G}_{F,\mathcal{X}} = n$ and so the sampling results of Lemma \ref{intro:lem} and Theorem \ref{thm:sensitivitySampling} are vacuous -- they require sampling $\ge n$ points, even for this simple task.}\label{fig:largeSensitivity}
\end{figure}

\subsection{Related Work}

The sensitivity sampling framework has been successfully applied to a number of problems, including clustering \citep{feldman2011unified,lucic2015strong, bachem2015coresets}, logistic regression \citep{huggins2016coresets,munteanu2018coresets}, and least squares regression, in the form of leverage score sampling \citep{drineas2006sampling,mahoney2011randomized,cohen2015uniform}. In these works, upper bounds are given on the sensitivity of each data point, and it is shown that the sum of these bounds, and thus the required sample size for approximate optimization, is small. We aim to expand the applicability of sensitivity-based methods to functions for which a bound on the sensitivity cannot be obtained or for which the total sensitivity  is inherently large. 

The local-sensitivity-based iterative method that we will discuss is closely related to quasi-Newton methods \citep{dennis1977quasi}, especially those that approximate the Hessian via leverage score sampling \citep{xu2016sub,ye2017approximate}. In each iteration, we estimate local sensitivities by considering the sensitivities of a local quadratic approximation to $F$. As shown in Section \ref{sec:levrage_score}, these sensitivities can be bounded using the leverage scores of the Hessian, and thus our sampling probabilities are closely related to those used in the above works. Unlike a quasi-Newton method however, we use the sensitivities to directly optimize $F$ locally, rather than the quadratic approximation itself. In this way, our method is closer to a trust region method \citep{chen2018stochastic} or an approximate proximal point method \citep{frostig2015regularizing}. 

Recently, \citep{agarwal2017leverage}  and \citep{chowdhury2018iterative} have suggested iterative algorithms for regularized least squares regression and ERM for linear models that sample a subset of data  points by their leverage scores (closely related to sensitivities) in each step. These works employ this sampling in a different way than us, using the subsample to precondition each iterative step. While they give strong theoretical guarantees for the problems studied, this technique applies to a less general class of problems than our method.

The sensitivity scores for $\ell_2$ regression are commonly known as leverage scores, and a long line of work \citep[see, e.g.,]{rudi2018fast,altschuler2018massively} has focused on approximating these scores more quickly. These approximation techniques do not extend to general sensitivity score approximation however. Additionally, our paper in no way attempts to develop a faster algorithm for leverage score sampling. We focus on introducing the notion of local sensitivity, which allows leverage score based methods to be applied to optimization problems well beyond $\ell_2$ regression.
\subsection{Road Map}
Our contributions are presented as follows. In Section \ref{sec:levrage_score} we show that the sensitivity scores of a quadratic approximation to a function are given by the leverage scores of an appropriate matrix. We  use these scores to bound the local sensitivity scores of the true function. In Section \ref{sec:sampling_via_local} we discuss how to subsample using these approximate local sensitivities with the aim of approximately minimizing the function over a small ball. We describe how to use this approach to iteratively  optimize the function. In Section \ref{sec:convex} we give an analysis of this iterative method for convex functions.

\section{Leverage Scores as Sensitivities of Quadratic Functions}\label{sec:levrage_score}

We start by showing how to  approximate  the local sensitivity $\sigma_{F,\mathcal{X} \cap B(r,y)}$ over some ball by approximating $F$ with a quadratic function on this ball. $F$'s sensitivities can be approximated by those of this quadratic function, which we in turn bound in closed form by the leverage scores of an appropriate matrix (a rank-$1$ perturbation of $F$'s Hessian at $y$). The leverage scores are given by:
\begin{definition}[Leverage Scores \citep{alaoui2015fast,cohen2017input}]\label{def:leverage}
For any $C \in \R^{n \times p}$ with $i^{th}$ row $c_i$, the $i^{th}$ $\lambda$-ridge leverage score is the sensitivity of $F(z) = \norm{Cz}_2^2 + \lambda \norm{z}_2^2$:
\begin{align*}
\ell_i^\lambda(C):= \max_{\{z \in \mathbb{R}^p : \norm{z}_2> 0\}} \frac{[Cz]_i^2}{\norm{Cz}_2^2 +\lambda \norm{z}_2^2}.
\end{align*}
We have $\ell_i^\lambda(C) = c_i^T (C^T C+\lambda I)^{-1}c_i$. (See Lemma \ref{lem:leverage} in Appendix \ref{ap:sec2}).
\end{definition}
Our eventual iterative method will employ a proximal function, and thus in this section we consider this function, which reduces to $F$ when $\lambda = 0$:
\begin{definition}[Proximal Function]\label{def:prox}
For a function $F: \mathcal{X} \rightarrow \R$, define $F_{\lambda,y}(x) = F(x) + \lambda \norm{x-y}_2^2$.
\end{definition}

Using Definition \ref{def:leverage} and the associated Lemma \ref{lem:leverage} we establish the following in Appendix~\ref{ap:sec2}. 
\begin{theorem}[Sensitivity of Quadratic Approximation]\label{thm:quad}
Consider $F$ as in Def. \ref{def:fs} along with the quadratic approximation to the proximal function $F_{\lambda,y}$ (Def. \ref{def:prox}) around $y \in \mathcal{X}$.
If $A \in \mathbb{R}^{n\times d}$ is the data matrix with  $i^{th}$ row equal to $a_i$, then 
\begin{align}
\begin{split} \label{eq:thm_quad}
    &\tilde F_{\lambda,y}(x) := \frac{1}{n} \sum_{i=1}^n \left [f_i(a_i^T  y) + a_i^T (x-y) \cdot f'(a_i^T y)  \right. \left.  + \frac{1}{2} (a_i^T (x-y))^2 \cdot f''(a_i^T y)\right ] + \gamma(x) + \lambda \norm{x-y}_2^2\\
    &:=F(y) + (x-y)^T A^T \alpha_y + \frac{1}{2} (x-y)^T A^T H_y A (x-y)   + \gamma(x) + \lambda \norm{x-y}_2^2
    \end{split}
\end{align}
where $[\alpha_y]_i = \frac{1}{n} f_i'(a_i^T y)$, and $H_y$ is the diagonal matrix with $[H_y]_{i,i} = \frac{1}{n} f''(a_i^T y)$. Assuming that $H_y$ is nonnegative,
the sensitivity scores of $\tilde F_{\lambda,y}$ with respect to $B(r,y)$ can be bounded as
\begin{align}\label{eq:sensUpperBound}
\sigma_{\tilde F_{\lambda,y}, B(r,y)}(a_i) &\le \beta \cdot \ell_i^\lambda(C) + \frac{f_i(a_i^T y)}{\eta},
\end{align}
where $C = [H_y^{1/2} A, \frac{1}{\delta}H_y^{-1/2} \alpha_y]$, $\ell_i^\lambda(C)$ is the leverage score of Def. \ref{def:leverage}, $\displaystyle \eta = \min_{x \in B(r,y)} \tilde F_{\lambda,y}(x)$, $\displaystyle \delta = \min_{x \in B(r,y)} \gamma(x)$, and $\beta = \max \bigg (1, 1 - \frac{F(y) - \frac{1}{n} \sum_{i=1}^n \frac{f'(a_i^T y)^2}{4 f''(a_i^T y)}}{\eta} \bigg )$.
\end{theorem}
Note that if we consider a small enough ball, where $\tilde F_{\lambda,y}$ well approximates $F_{\lambda,y}$, we expect $\displaystyle \eta =  \min_{x \in B(r,y)} \tilde F_{\lambda,y}(x)  = \Theta(F(y))$. Thus, the additive $\frac{f_i(a_i^T y)}{\eta}$ term on each sensitivity will contribute only a $\frac{\sum f_i(a_i^T y)}{\Theta(F(y))} = O(1)$ additive factor to the total sensitivity bound and sample size.

\subsection{Efficient Computation of Leverage Score Sensitivities}
The sensitivity upper bound \eqref{eq:sensUpperBound} of Theorem \ref{thm:quad} can be approximated efficiently as long as we can efficiently approximate the leverage scores 
$
\ell_i^\lambda(C) = c_i^T (C^T C + \lambda I)^{-1} c_i, $
where $C = [H_y^{1/2} A, \frac{1}{\delta}H_y^{-1/2} \alpha_y]$. We can use a block matrix inversion formula to find that
\begin{align*}
&(C^T C + \lambda I)^{-1} = \begin{bmatrix} A^T H_y A + \lambda I & \frac{1}{\delta}A^T \alpha_y \\ \frac{1}{\delta}\alpha_y^T A & \norm{\alpha_y}_2^2 + \lambda \end{bmatrix}^{-1}  \  = \begin{bmatrix} A_1 & A_2 \\ A_2^\top & \frac{1}{k} \end{bmatrix}
\end{align*}
where  $A_1  =(A^T H_y A + \lambda I)^{-1} + \frac{1}{k} (A^T H_y A + \lambda I)^{-1}  A^T \alpha_y \alpha_y^T A (A^T H_y A + \lambda I)^{-1}$, $k = \norm{\alpha_y}_2^2 + \delta^2 \lambda - \alpha_y^T A(A^T H_y A + \lambda I)^{-1} A^T \alpha_y$ ,   and  $A_2=- \frac{\delta}{k} (A^T H_y A + \lambda I)^{-1} A^T \alpha_y$.  

Thus, if we have a fast algorithm for applying $(A^T H_y A + \lambda I)^{-1}$ to a vector we can quickly apply  $(C^T C + \lambda I)^{-1} $ to a vector and compute the leverage scores $\ell_i^\lambda(C) = c_i^T (C^T C + \lambda I)^{-1} c_i$. Via standard Johnson-Lindenstrauss sketching techniques \citep{spielman2011graph}  it in fact suffices to apply this inverse to $O(\log n/\delta)$ vectors to approximate each score up to constant factor with probability $\ge 1-\delta$. In practice, one can use traditional iterative methods such as conjugate gradient,  iterative sampling methods such as those presented in \citep{cohen2015uniform,cohen2017input}, or  fast sketching methods \citep{drineas2012fast, clarkson2017low}.

\subsection{True Local Sensitivity from Quadratic Approximation}\label{sec:local}
As long as the quadratic approximation $\tilde F_{\lambda,y}$ approximates $F_{\lambda,y}$ sufficiently well on the ball $B(r,y)$, we can use Theorem \ref{thm:quad} to approximate the true local sensitivity ${\sigma}_{ F_{\lambda,y},\mathcal{X} \cap B(r,y)}(a_i)$. We start by discussing our approximation assumptions. 

Defining $\alpha_y$ as in Theorem \ref{thm:quad}, for some $B_y(x)$ which itself is a function of $x$ we have:
\begin{align*}
F(x) &= F(y) + (x - y)^\top A^\top \alpha_y + (x - y)^\top  A^\top H_y A (x - y)  + \gamma(x) + B_y(x)\| x -y \|_2^{3}. 
\end{align*}
Without loss of generality, we assume that $B_y(x)>0$ for $x$ in the above equation or we just shift the overall function vertically by adjusting $\gamma(\cdot)$ to have the quadratic appropriator be an under approximation of the true function. If the function $F$ has a $C$ Lipschitz-Hessian then we have:
\begin{align}\label{eq:cApprox}
F(x) &\leq F(y) + (x - y)^\top A^\top \alpha_y + (x - y)^\top  A^\top H_y A (x - y)  + \gamma(x) +\frac{C}{6} \|x - y \|_2^3.
\end{align}
  
For simplicity, we also assume that \eqref{eq:cApprox} holds componentwise with Lipschitz Hessian constant $C_i$ for $i \in [n]$. Adding the second order approximation of $F(x)$ to $\lambda \|x - y \|_2^2$ gives the approximate function $\tilde F_{\lambda,y} (x)$ as defined in ~\eqref{eq:thm_quad}. Theorem~\ref{thm:quad} shows how to bound the sensitivities of $\tilde F_{\lambda,y} (x)$. Using \eqref{eq:cApprox} we prove a bound on the local sensitivities of $F_{\lambda,y} (x)$ itself in Appendix~\ref{ap:true_local}:
\begin{theorem}\label{thm:local_true_sens}
Consider $F_{\lambda,y}$ as in Defs. \ref{def:fs}, \ref{def:prox}, $y \in \mathcal{X}$, a radius $r$, and $\displaystyle \alpha = \min_{x \in B(r,y)} F_{\lambda,y}(x)$. Then, $\forall ~ i \in [n]$,
\begin{align*}
{\sigma}_{ F_{\lambda,y},B(r,y)}(a_i) \le {\sigma}_{\tilde F_{\lambda,y},B(r,y)}(a_i) + \min \left (\frac{C_i r}{6n\lambda}, \frac{C_i r^3}{6n \alpha}  \right).
\end{align*}
  \end{theorem}

Using this sensitivity  bound, we can independently sample components with the computed scores as in Definition \ref{def:subSampled}, obtaining a  $(1+\epsilon)$ approximation of the function ${F}_{\lambda,y}(x)$. That is, letting $ F_{\lambda,y}^s (x)$ represent the subsampled empirical loss function (sampled as in Theorem \ref{thm:sensitivitySampling}), for $\tilde O \left (\frac{\Delta}{\epsilon^2}\right )$ samples, we have $ F_{\lambda,y}^s (x) \in (1\pm\epsilon) F_{\lambda,y} (x)~\forall ~x \in{B}(y,R)$ with high probability.

\section{Optimization  via Local Sensitivity Sampling}\label{sec:sampling_via_local}
In Theorem \ref{thm:local_true_sens} we showed how to bound the local sensitivities of a function $F :=  \sum_{i=1}^n f_i(a_i^T  x) + \gamma(x)$ using the local sensitivities of a quadratic approximation to $F$, which are given by  the leverage scores of an appropriate matrix (Theorem \ref{thm:quad}). These sensitivities are only valid in a sufficiently small ball around some starting point $y$, roughly, where the quadratic approximation is accurate. In this section we show how they can be used to optimize $F$ beyond this ball, specifically as part  of an iterative method that locally optimizes $F$ until convergence to a global optimum.

In the optimization literature, there are two popular techniques that iteratively optimize a function via local optimizations over a ball: (i) trust region methods \citep{conn2000trust} and (ii) proximal point methods \citep{parikh2014proximal}. 
Local sensitivity sampling can be combined with both of these classes of methods. We first focus on proximal point methods, discussing a related trust region approach in Section \ref{sec:trust}.
 In the proximal point method, the idea is in each step to approximate a regularized minimum:
\begin{align}
\begin{split} \label{eq:ppm_optima}
x^\star_{\lambda_t,y} = \argmin F_{\lambda_t,y}(x)& = \argmin \left[ F(x) + {\lambda_t} \| x -y \|_2^2 \right] \\
 \text{ and } F^{\star}_{\lambda_t,y} &= F_{\lambda_t,y}(x_{\lambda_t,y}^\star).
 \end{split}
\end{align}
Here $\lambda_t$ is a regularization parameter depending on the iteration $t$.
As discussed below,  minimizing this regularized function is equivalent to minimizing $F$ on a ball of a given radius.

\subsection{Equivalence between Constrained and Penalized Formulation}
When $F$ is convex it is well known that for any $\lambda$ minimizing the proximal function $F_{\lambda,y}$ is equivalent to minimizing $F$ constrained to some ball around $y$. 
Consider the constrained optimization problem given in equation~\eqref{eq:ppm_equation_constrained} where ${B}({r},y)$ is the ball of radius ${r}$ centered at $y$:
\begin{align}
\begin{split}
 x_{r,y}^\star = \argmin_{x \in {B}(r,y)}     F(x)  . \label{eq:ppm_equation_constrained}
\end{split}
\end{align}
  
\begin{lemma}\label{lem:Equivalence}
Let  $x^\star = \argmin_{x \in \mathbb{R}^d}     F(x) $ for a convex function $F$. If $x^\star$ does not lie inside ${B}(r,y)$ then $x_{r,y}^\star$ also solves the following optimization problem:
\begin{align}
x_{{ r},y}^\star = \argmin_{x \in \mathbb{R}^d} ~ F(x) +  \frac{\|\nabla F(x_{{ r},y}^\star) \|}{2{ r}} \cdot  \| x - y \|_2^2   \label{eq:penalized_ppm_eq}.
\end{align}
\end{lemma}
Comparing equations~\eqref{eq:ppm_optima} and \eqref{eq:penalized_ppm_eq}, se see that  $\lambda = \frac{\|\nabla F(x_{r,y}^\star) \|}{{ 2r}} \Rightarrow { r} = \frac{\|\nabla F(x_{r,y}^\star) \|}{2\lambda}$  .  While it is not directly possible to compute radius $r$ in closed form without computing $x_{{ r},y}^\star$ itself, we can give a computable upper bound on $r$ which will be crucial for our analysis. 

\begin{lemma} \label{lem:radius_bound}
Consider the optimization problem~\eqref{eq:ppm_optima} and its corresponding constrained  counterpart ~\eqref{eq:ppm_equation_constrained} where $F$ is a $\mu$ strongly convex function. Then, $x_{\lambda,y}^\star$ falls within a ball of radius $r = \frac{\| \nabla F(y) \|}{2\lambda + \mu}$ around $y$.
\end{lemma}
Proofs for this sections are provided in the Appendix~\ref{app:constPenalized}.

Using the local sensitivity  bound of Section \ref{sec:local} we can approximate $F_{\lambda,y}$ on a ball of small enough radius. In applying sensitivity sampling to a proximal point method, it will be critical to ensure that $\lambda_{t}$ is not too small. This will ensure that, by Lemma \ref{lem:radius_bound}, $x^\star_{\lambda_y}$ falls in a sufficiently small  radius, and so  an approximate minimum can be found via local sensitivity sampling.

\subsection{Algorithmic Intuition}
By Theorem \ref{thm:sensitivitySampling} if we subsample the proximal function $F_{\lambda_t,y}$ using the local sensitivity bound of Theorem \ref{thm:local_true_sens} for a sufficiently large radius $r$ (as a function of $\lambda_t$ via Lemma \ref{lem:radius_bound}), optimizing this function will return a value within a $1+\epsilon$ factor of the true minimum $x^\star_{\lambda_t,y}$ with high probability. Abstracting away the sensitivity sampling technique, our goal becomes to analyze the convergence of the approximate proximal point method (APPM) when the optimum is computed up to $1+\epsilon$ error in each iteration. We give pseudocode for this general method in Algorithm \ref{algo:approximate_ppm}.

\begin{algorithm}[H]
\begin{algorithmic}[1] 
  \STATE  \textbf{input} $x_0 \in \mathbb{R}^d$, $\lambda_t > 0 ~\forall{t\in [T]}$.
  \STATE  \textbf{input} Black-box $\epsilon$-oracle $\mathcal{P}_{F_{\lambda_1,x_0}}$
  \STATE \textbf{for}~ {$t=1\dots \text{ T}$} \textbf{do}
  \STATE \quad  $x_t \gets {P}_{F_{\lambda_t,x_{t-1}}}(x)$ \label{alg_lin:exact_g_lin3}
    \STATE \textbf{end for}
  \STATE \textbf{output} $x_T$ 
\end{algorithmic}
 \caption{APPM}
 \label{algo:approximate_ppm}
\end{algorithm}
\begin{definition}
An algorithm $\mathcal{P}_f$ is called \textit{multiplicative }$\epsilon$-oracle for a given function $F$  if $F(x^\star) \leq F\big(\mathcal{P}_F(x)\big) \leq (1+\epsilon) F(x^\star)$ where $x^\star$ if the true minimizer of $F$.
\end{definition}
In Algorithm~\ref{algo:approximate_ppm}, we provide the pseudocode for APPM under the access of a \textit{multiplicative }$\epsilon$-oracle at each iterate. In our setting, $\mathcal{P}_F$ employs local sensitivity sampling.

\section{Convergence Analysis for Smooth Convex Functions}\label{sec:convex}

In this section, we analyze the convergence of Algorithm~\ref{algo:approximate_ppm} with an $\epsilon$ oracle obtained via local sensitivity sampling.  We demonstrate how to set the regularization parameters $\lambda_t$ in each step and then in the end provide a complete algorithm. Let $F^\star$ denote $F(x^\star)$. Throughout we make the following assumption about $F(x)$:
\begin{itemize}
   \item $F$ is $\mu$-strongly convex, \textit{i.e.}, for all $x,y \in \mathbb{R}^d$, $        F( y) \geq F (x) + \langle \nabla F(x), y - x \rangle + \frac{\mu}{2}\| y - x \|_2^2.  $
\end{itemize}

\subsection{Approximate Proximal Point Method with {Multiplicative} Oracle} \label{sub:approximate_ppm}
We first state convergence bounds for Approximate Proximal Point Method  (Algorithm \ref{algo:approximate_ppm}) with a blackbox multiplicative $\epsilon$-oracle. Our first bound assumes strong convexity, our second does not. Proofs are given in Appendix~\ref{ap:app_ppm}.

\begin{theorem} \label{thm:approximate_ppm_coreset}
For $\mu$-strongly convex $F$, consider $\epsilon_{1},\ldots \epsilon_T \in (0,1)$ and $x_{0},\ldots, x_{T}\in \mathbb{R}^d$ such that $x_t = {P}_{F_{\lambda_t,x_{t-1}}}(x_{t-1})$ where ${P}_{F_{\lambda_t,x_{t-1}}}$ is an $\epsilon_t$-oracle (see Algorithm \ref{algo:approximate_ppm}). 
 Then if $\epsilon_t \leq \frac{\mu}{\mu + \lambda_t}  \forall t\in [T]$, we have  $ F(x_t)    - F^\star  \leq {\frac{1}{1 - \epsilon_t}} \frac{\lambda_t}{\mu+\lambda_t} \left( F(x_{t-1}) - F^\star \right)  + {\frac{\epsilon_t}{1 - \epsilon_t}} F^\star ~\forall t \in [T]$    and 
  $$ F(x_T)  - F^\star \leq \rho (F(x_0) - F^\star) + \delta F^\star $$  
where $\rho = \prod_{t=1}^T {\frac{1}{1 - \epsilon_t}} \frac{\lambda_t}{\mu+\lambda_t}$ and $\delta =  {\mathlarger{\sum}}_{t=1}^T \Big({\frac{\epsilon_t}{1 - \epsilon_t}} \prod_{j = t+1}^T {\frac{1}{1 - \epsilon_t}} \frac{\lambda_j}{\mu + \lambda_j}\Big)$.
\end{theorem}
\begin{theorem}\label{thm:approximate_ppm_coreset_smooth}
For a  smooth convex function $F$,  let $\epsilon_{1},\ldots,\epsilon_{T} = \epsilon \text{ where } \epsilon \in (0,1/2)$ and $x_0,\ldots,x_T \in \mathbb{R}^d$ be as in Theorem \ref{thm:approximate_ppm_coreset}.
Then, we have  
  $$ F(x_T) - F^\star  \leq \frac{2}{(1-\epsilon)} \frac{\| x^\star - x_{0} \|_2^2}{\sum_{t=1}^T \frac{2}{\lambda_t}} + \frac{3\epsilon}{1-\epsilon} F^\star.$$   
\end{theorem}

\subsection{Local Sensitivity Sampling} \label{subsec:local_sens_analysis} 
We now discuss how to choose the parameters for Algorithm~\ref{algo:approximate_ppm} when using local sensitivity sampling to implement the $\epsilon$-oracle in each step. From Lemmas~\ref{lem:Equivalence} and \ref{lem:radius_bound} it is clear that if $\lambda_t$ goes down, the corresponding radius $r_t$ goes up.  However, in Theorem~\ref{thm:local_true_sens}, we bound the true local sensitivity at iteration $t$ by a quantity depending on $\frac{r_t}{\lambda_t}$, which comes from the error in the quadratic approximation. Thus, if we choose $\lambda_t$ very small, the term $\frac{r_t}{\lambda_t}$ will dominate in the local sensitivity approximation, and we won't see any advantage from local sensitivity sampling over, e.g., uniform sampling. Making $\lambda_t$ large will improve the local sensitivity  approximation but slow down convergence. 

To balance these factors, we will  choose $\lambda_t$ of the order of $r_t$. In particular, considering Lemma \ref{lem:radius_bound}, we choose $\lambda_t = \sqrt{\|\nabla F(x_{t-1})\|}_2$. The lemma then gives that $r_t \leq \frac{\|\nabla F(x_{t-1}) \|}{\sqrt{\|\nabla F(x_{t-1})\|} + \mu} \leq \sqrt{\|\nabla F(x_{t-1})\|}_2$. We here now provide an end to end algorithm which utilizes local sensitivity sampling in the approximate proximal point method framework presented in Algorithm~\ref{algo:approximate_ppm}. The pseudo-code and details of the algorithm are given in Algorithm~\ref{algo:end_to_end_appm} where we denote ${F}_{\lambda_t,x_{t-1}}^s(x)$ as the importance sampled subset of ${F}_{\lambda_t,x_{t-1}}(x)$ which has been obtained via local sensitivity sampling. Line 9 of Algorithm ~\ref{algo:end_to_end_appm}  can be considered as a black-optimization problem which is apparently a strongly-convex optimization problem and can be optimized exponentially fast.

\paragraph{On Convergence:} With this choice of $\lambda_t$, the convergence rate of APPM under our strong convexity assumption will be $ \mathcal{O}\left( \frac{\| \sqrt{ \tilde \nabla F(x)}\|_2}{\mu} \log (1/\varepsilon) \right) $ where $\sqrt{\| \tilde \nabla F(x)\|}_2$ represents  $\frac{1}{T} \sum_{i =0}^{T-1} \sqrt{ \| \nabla F(x_i)  \|}_2$. If $F$ is smooth with smoothness parameter $L$, we have: $\|\nabla F(x) \|_2\leq L \|x - x^\star \|_2$. For the smooth but non-strongly convex problem, if we assume $\lambda_t \le \epsilon$ for some $\epsilon$ for all $t$ then, $\|\nabla F(x_t) \|_2^2 \in \mathcal{O}(1/{T})$ in the worst case. Hence, the rate of for non-strongly convex smooth function will behave like $\mathcal{O}(1/T^{5/4})$.

\begin{algorithm}
\begin{algorithmic}[1] 
  \STATE  \textbf{input} $x_0 \in \mathbb{R}^d$, $\epsilon_t$, and $\mu$.
  \STATE Compute $\|\nabla F(x_0)\|_2$, $F(x_0)$, and $C_0$
  \STATE \textbf{for}~ {$t=1\dots \text{ T}$} \textbf{do}
  \STATE \quad Compute regularizer $\lambda_t$ $\gets$ $\sqrt{\|\nabla f(x_{t-1})\|}_2$. \label{alg_lin:exact_g_lin1_end}
  \STATE \quad Compute radius $r_t$ $\gets$ $\frac{\|\nabla f(x_{t-1})\|_2}{\sqrt{\|\nabla f(x_{t-1})\|}_2+ \mu}$. \label{alg_lin:exact_g_lin2_end}
  \STATE \quad Get $\tilde F_{\lambda_{t}, x_{t-1}}$ via Taylor Expansion.\label{alg_lin:exact_g_lin3_end}
  \STATE \quad Compute the local sensitivity for $ F_{\lambda_{t}, x_{t-1}}$ using Theorem \ref{thm:local_true_sens}.\label{alg_lin:exact_g_lin4_end}
  \STATE \quad Local sensitivity based sampling of ${F}_{\lambda_t,x_{t-1}}^s(x)$ from ${F}_{\lambda_t,x_{t-1}}(x)$.\label{alg_lin:exact_g_lin4_end}
   \STATE \quad $x_{t} \gets \argmin_{x \in B(r_t,x_{t-1})} F_{\lambda_t,x_{t-1}}^s(x)$.\label{alg_lin:exact_g_lin5_end_opt}
    \STATE \quad  Compute $\|\nabla F(x_t)\|_2$.\label{alg_lin:exact_g_lin8_end}
    \STATE \textbf{end for}
  \STATE \textbf{output} $x_T$ 
\end{algorithmic}
\caption{APPM with Local Sensitivity Sampling}
 \label{algo:end_to_end_appm}
\end{algorithm}

\section{An Adaptive Stochastic Trust Region Method}\label{sec:trust}

Related to the proximal point approach, sensitivity sampling can be used to obtain an adaptive stochastic trust region. In each iteration $t$, we approximately minimize a quadratic approximation to  $F$ over a ball, using local sensitivity sampling and directly applying the sensitivity score bound of Theorem \ref{thm:quad}. At iteration $t$ the center of the ball is at 

 $x_{t-1}$ and the radius is set to $r_t = \frac{\|\nabla F(x_{t-1}) \|_2}{\lambda_{t}+ \mu}$. We provide pseudocode in Algorithm \ref{algo:ada_trust_region} and a proof of a convergence bound in Appendix~\ref{ap:quad_approx}. Here we just state the main result. 
\begin{theorem} \label{thm:main}
For a given set of constants $C_k $, $\delta_k \in (0,1)$,  and $\tilde \epsilon_k  = \delta_k \frac{\mu}{\lambda_k + \mu}$ which is an error tolerance for the quadratic approximation of the function $F_{\lambda_{k},x_{k-1}}(x)$  for all $k$, if $\lambda_{k+1}$ is chosen of $\mathcal{O}(\sqrt{\|\nabla F(x_k) \|_2})$ then at iteration $k+1$ Algorithm \ref{algo:ada_trust_region} satisfies:
\begin{align}
F(x_{k+1})  - F^\star &\leq (1 + 2 \epsilon_{k+1}) \frac{2\lambda_{k+1}}{2\lambda_{k+1}+\mu} \left(F(x_k) - F^\star \right)  + 2 \epsilon_{k+1} F^\star, \label{eq:main_thm_eq1_paper}
\end{align}
where $\epsilon_{k+1} = 2\tilde{\epsilon}_{k+1}\left( 1  + \frac{1}{m} \right)$, $m$ and $c$ are positive constants. 
\end{theorem}

Comparing equation~\eqref{eq:main_thm_eq1_paper} in Theorem~\ref{thm:main} with the bound in Theorem~\ref{thm:approximate_ppm_coreset}, we can see that we have obtained a similar recursive relation in both equations, and hence the trust region method will have a similar convergence rate to APPM in the presence of an $\epsilon$-\textit{multiplicative} oracle.

\begin{figure*}
	\begin{subfigure}{.325\textwidth}
    	\centering
    	\includegraphics[width=.99\linewidth]{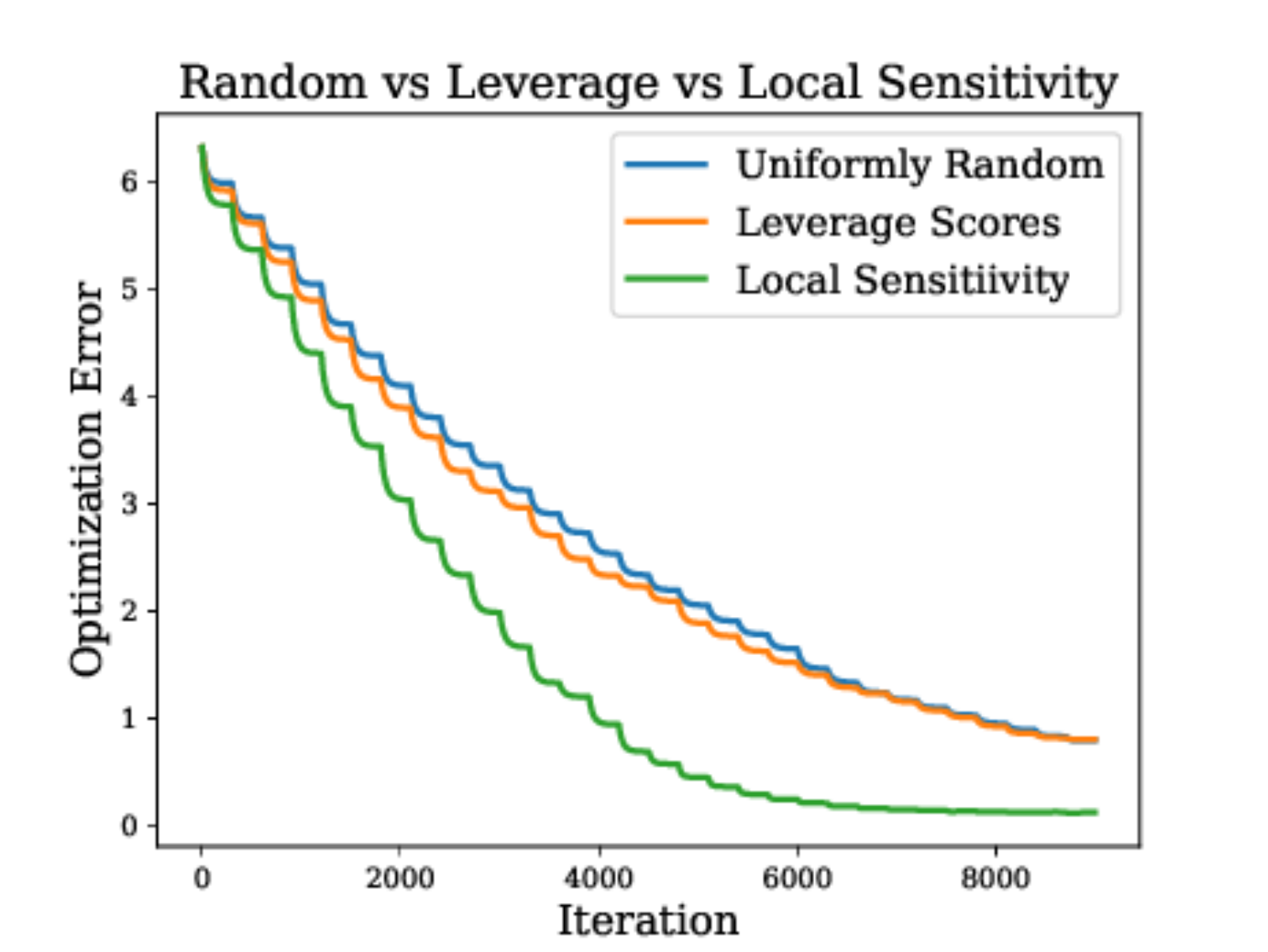}
    	\caption{\textit{Synthetic Data}}
    	\label{fig:select_fig4}
  \end{subfigure}	
  \begin{subfigure}{.325\textwidth}
    \centering
    \includegraphics[width=.99\linewidth]{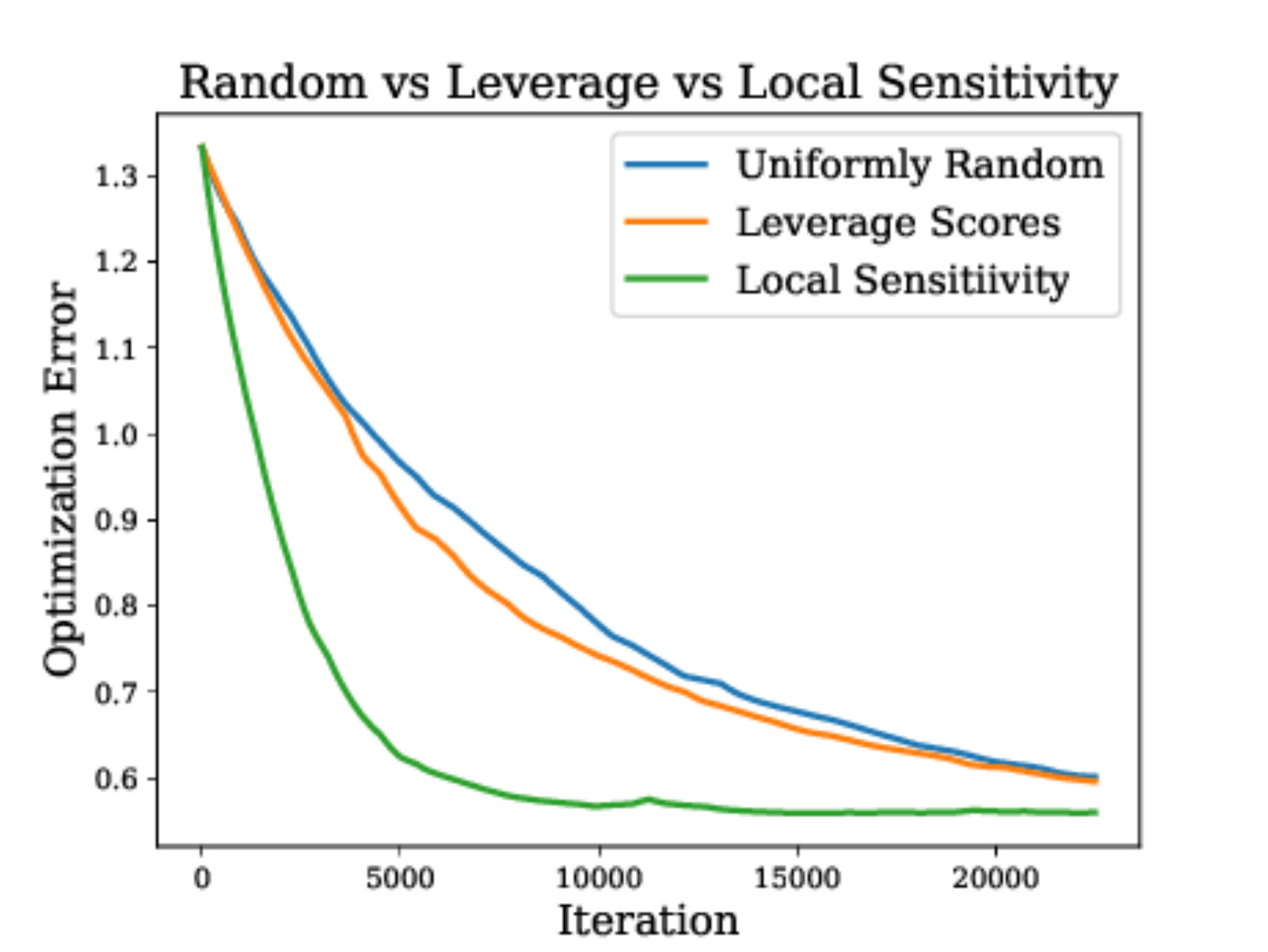}
    \caption{\textit{Letter-Binary Train} }
    \label{fig:select_fig1}
  \end{subfigure}%
  \begin{subfigure}{.325\textwidth}
    \centering
    \includegraphics[width=.99\linewidth]{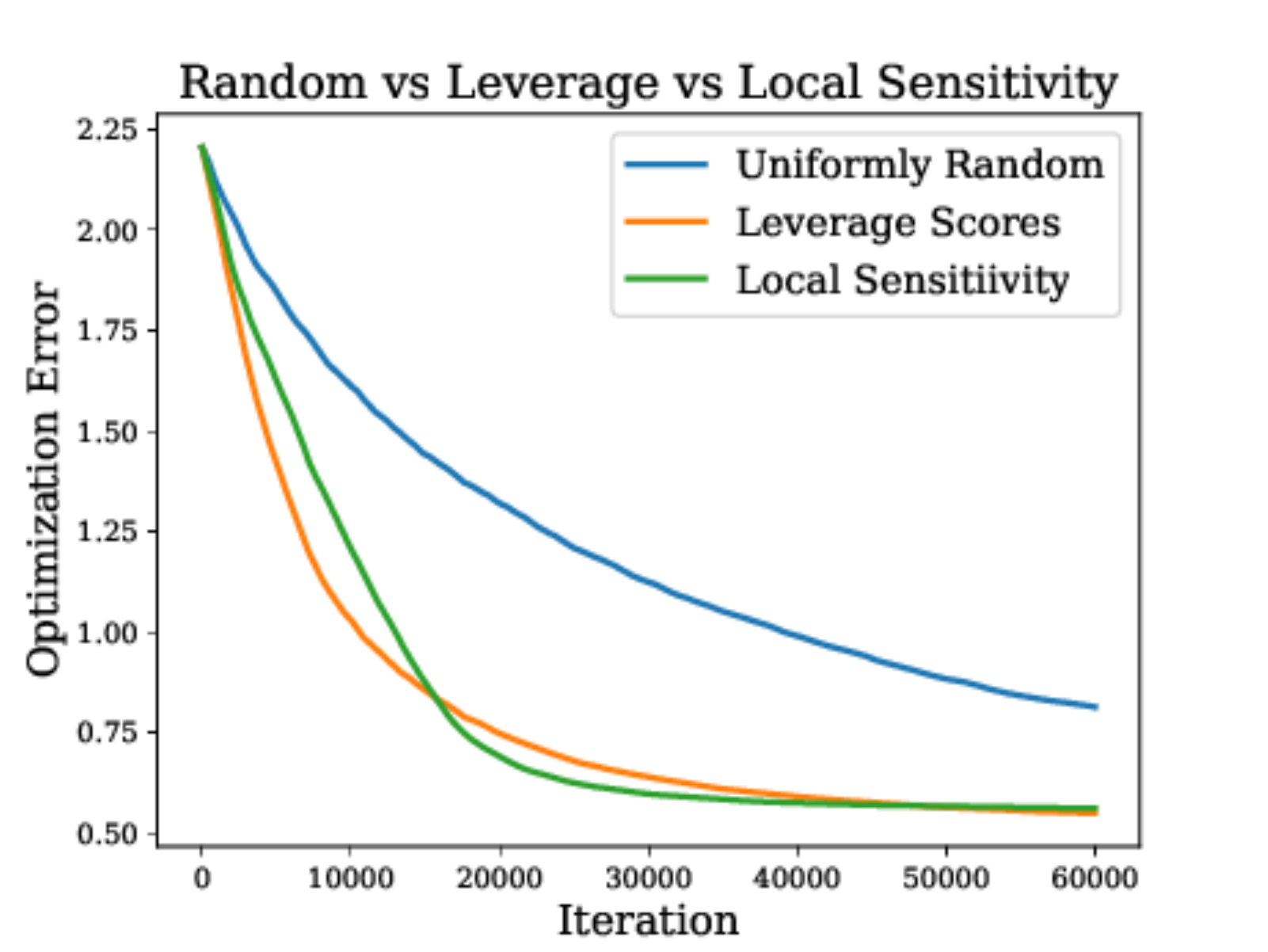}
    \caption{\textit{Magic04 Test}}
    \label{fig:select_fig2}
  \end{subfigure}
  \begin{subfigure}{.325\textwidth}
    \centering
    \includegraphics[width=.99\linewidth]{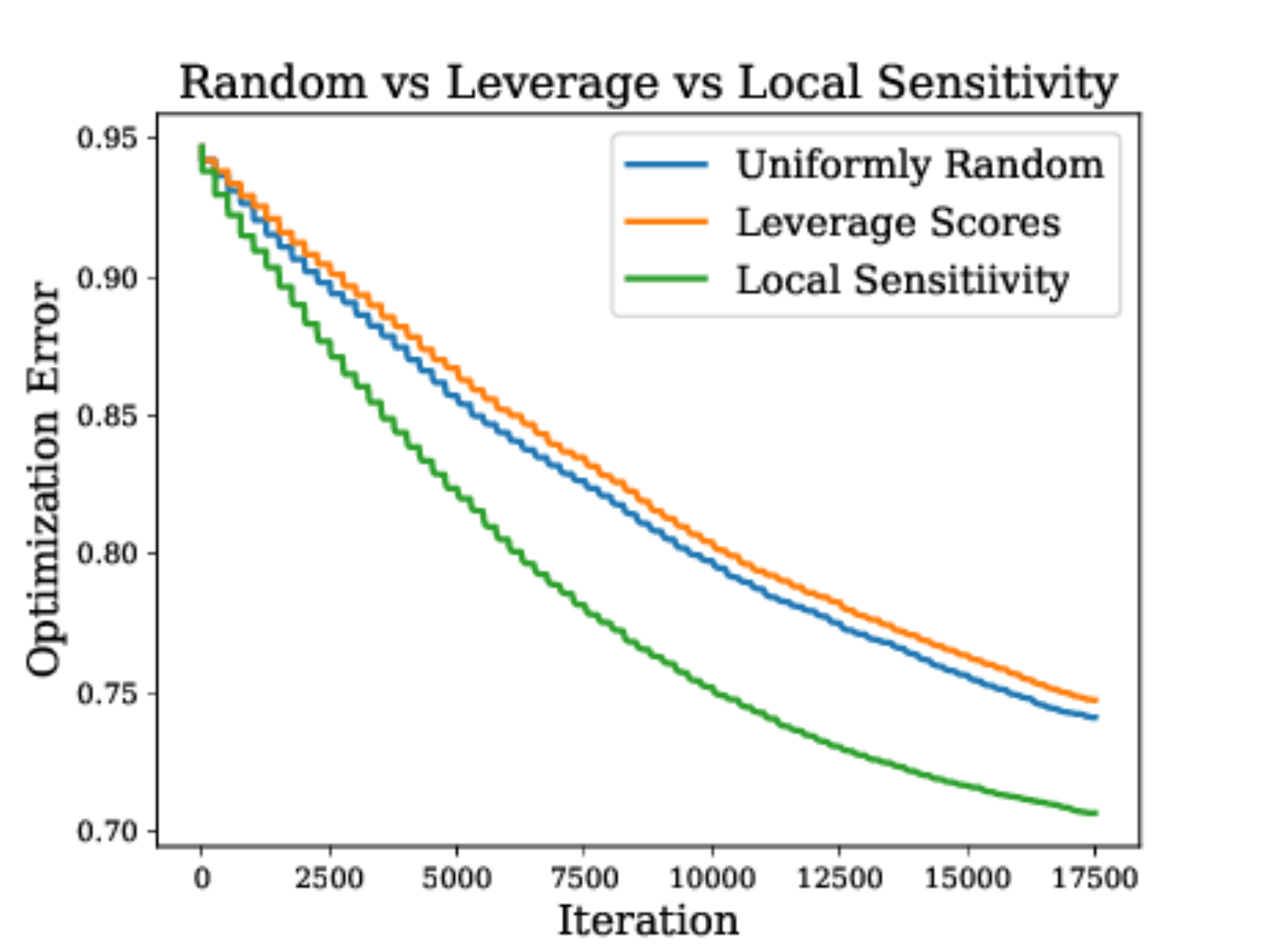}
    \caption{\textit{MNIST Test}}
    \label{fig:select_fig3}
  \end{subfigure}%
\begin{subfigure}{.325\textwidth}
    	\centering
    	\includegraphics[width=.99\linewidth]{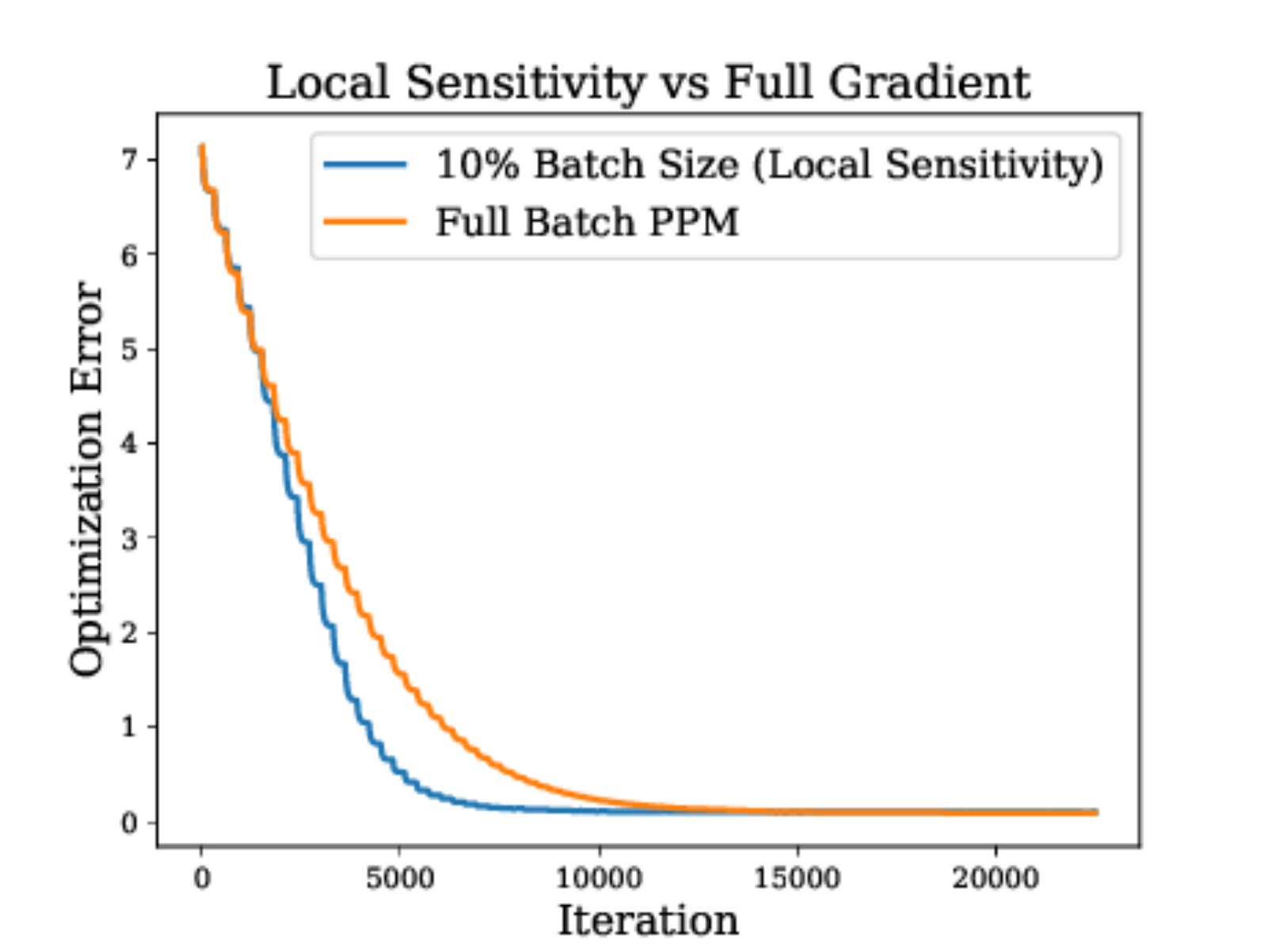}
    	\caption{\textit{Synthetic Data}}
    	\label{fig:select_m_fig4}
  \end{subfigure}	
  \begin{subfigure}{.325\textwidth}
    \centering
    \includegraphics[width=.99\linewidth]{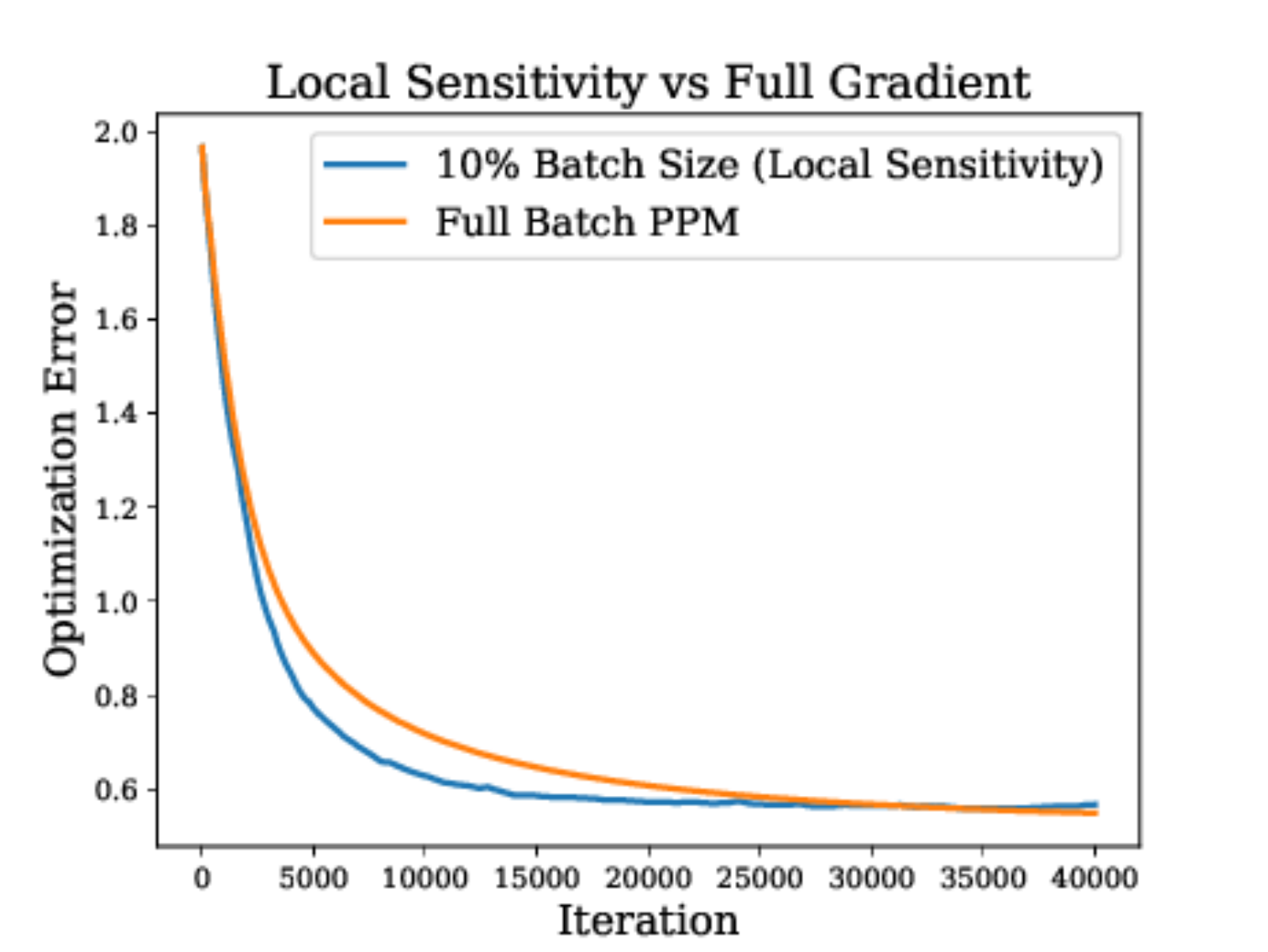}
    \caption{\textit{Letter-Binary Test} }
    \label{fig:select_m_fig1}
  \end{subfigure}%
  \caption{ (a-d) Local sensitivity sampling vs. uniform random sampling and leverage score sampling on four datasets: (a) \textit{Synthetic Data} (3000 points), (b) \textit{Letter Binary Train} (12000 points), (c) \textit{Magic04 Test} (4795 points), and (d) \textit{MNIST Test} (10000 points). (e-f) Local Sampling Method is compared  with Full Batch Gradient for (e) \textit{Synthetic } and (f) \textit{Letter Binary Test}.}
  \label{fig:selectfig}
\end{figure*}
\section{Experiments} \label{sec:expts}
We conclude by giving some initial experimental evidence to justify the performance of our proposed algorithm in practice. We provide the experiments for \textit{Approximate Proximal Point Method with Local Sensitivity Sampling} (Algorithm \ref{algo:end_to_end_appm}). We run our algorithm on the following four datatsets\footnote{Datasets can be downloaded from: \href{http://manikvarma.org/code/LDKL/download.html}{manikvarma.org/code/LDKL/download.html}.}
: (a) \textit{Synthetic Data} (b) \textit{Letter Binary} \citep{frey1991letter} (c) \textit{Magic04} \citep{bock2004methods} and (d) \textit{MNIST Binary} \citep{lecun1998gradient}. Prefix `Train' or `Test' denotes if the train or test split was used for the experiment. The \textit{Synthetic Data} was generated by first generating a matrix $A$ of size $3000 \times 300$ drawn from a 300 dimensional standard normal random variable. Then another vector $x_0$ of size 300 was fixed which is also drawn from a normal random variable to obtain $\hat{y} = Ax_0 + \eta$ where $\eta\sim 0.1*\mathcal{N}(0,1)$. Finally, the classification  label vector $y$ was chosen as $\text{sign}(\hat{y})$. We perform all our experiments for logistic regression with an $\ell_2^2$ regularization parameter of 0.001. For the experiments plotted in the Figure~\ref{fig:selectfig}, we have considered a fixed sample size of 100 data points for every iteration of the proximal algorithm. 
In the first four subfigures of Figure~\ref{fig:selectfig}, we compare compare local sensitivity sampling with two base lines: uniform random sampling and sampling using the leverage scores of the data matrix $A$. On the horizontal axis, we report the total number of iterations which is the number of times the sampling oracle is called (outer loop in Algorithm~\ref{algo:end_to_end_appm}) multiplied by number of times the gradient call to solve the optimization problem given in Line~9 in Algorithm~\ref{algo:end_to_end_appm}. We report the optimization error on vertical axis. 

From the plots in Figures~\ref{fig:select_fig4}, \ref{fig:select_fig1}, \ref{fig:select_fig2} and \ref{fig:select_fig3}, it is evident  that our method outperforms uniform random sampling with a large margin on the synthetic and real datasets. It also often performs much better than leverage score sampling. Since the local sensitively approximations of Theorems \ref{thm:quad} and \ref{thm:local_true_sens} are the leverage scores of a matrix with essentially  the same dimensions as $A$, these methods have the same order of computational cost.

We perform a second set of experiments to compare our sampling technique with  full batch gradient iteration for each proximal point iteration on \textit{Synthetic} and \textit{Letter Binary Test} which we plot in Figures~\ref{fig:select_m_fig4} and \ref{fig:select_m_fig1}. We can see in Figures~\ref{fig:select_m_fig4} and \ref{fig:select_m_fig1} that our sampling method outperforms the full gradient just with $10\%$ of total points. In both plots, the sampling method needs just half of the number of iterations taken by full gradient to saturate to similar value. 

In both of the experiments, we set the number of inner loop iteration (number of calls to the gradient oracle for solving Line~9 in Algorithm~\ref{algo:end_to_end_appm}) in advance to let the optimization error saturate for that particular outer loop; however the plots demonstrate that it can be set to a much smaller number or can be set adaptively to achieve gains of multiple folds. 
\section{Conclusion}
In this work, we study how the elegant approach of function approximation via sensitivity sampling can be made practical. We overcome two barriers: (1) the difficulty of approximating the sensitivity scores and (2) the high sample complexities required by  theoretical bounds. We handle both by considering a \emph{local} notion of sensitivity, which we can efficiently approximate and bound. We demonstrate that this notion can be combined with methods that globally optimize a function via iterative local optimizations, including proximal point and trust region methods.

Our work leaves open a number of questions. Most importantly, since local sensitivity approximation incurs some computational overhead (a leverage score computation along with some derivative computations), we believe it will be especially useful for functions that are difficult to optimize, e.g., non-strongly-convex functions. Understanding how our theory extends and how our method performs in practice on such functions would be very interesting. It would be especially interesting to compare performance to related approaches, such as quasi-Newton and other trust region approaches.

\clearpage

\newpage

\bibliography{opt-ml}
\bibliographystyle{abbrvnat}

\newpage
\onecolumn
\appendix
\begin{center}
{\centering \LARGE Appendix }
\vspace{1cm}
\sloppy

\end{center}

\section{Leverage Scores as Sensitivities of Quadratic Functions} \label{ap:sec2}

We here start by stating Lemma~\ref{lem:leverage} and giving its proof. This lemma is helpful in proving Theorem~\ref{thm:quad}. Lemma~\ref{lem:leverage} is a relatively well known characterization of the leverage scores of a matrix, see e.g, \cite{avron2017random}; however  for completeness we give a proof here.
\begin{lemma}[Leverage Scores as Sensitivities]\label{lem:leverage}
For any $C \in \R^{n \times p}$ with $i^{th}$ row $c_i$,
\begin{align*}
\ell_i^\lambda(C) = \max_{\{z \in \mathbb{R}^p : \norm{z} > 0\}} \frac{[Cz]_i^2}{\norm{Cz}_2^2 +\lambda \norm{z}_2^2} =  c_i^T (C^T C+\lambda I)^{-1}c_i.
\end{align*}
\end{lemma}
\begin{proof}
Write $\sigma(z) = \frac{[Cz]_i^2}{\norm{Cz}_2^2 +\lambda \norm{z}_2^2}$, $f(z) = [Cz]_i^2 = (c_i^T z)^2$, $g(z) = \norm{Cz}_2^2 +\lambda \norm{z}_2^2 = z^T(C^TC+\lambda I)z$. We can compute the gradient of $\sigma(z)$ as:
\begin{align*}
    \nabla_j \sigma(z) = \frac{\nabla_j f(z) \cdot g(z) - \nabla_j g(z) \cdot f(z)}{g(z)^2}.
\end{align*}
At the minimium this must equal $0$ and so since $g(z) > 0$ for $z$ with $\norm{z}_2 > 0$, we must have $\nabla f(z) \cdot g(z) - \nabla g(z) \cdot f(z) = 0$.
We have $\nabla f(z) = 2c_i^T z \cdot c_i$ and $\nabla g(z) = 2(C^T C + \lambda I)z$. We thus have at optimum:
\begin{align*}
    c_i \cdot \left (2c_i^T z \cdot z^T (C^TC+\lambda I)z\right ) - 2(C^TC+\lambda I)z \cdot (c_i^T z)^2  = 0.
\end{align*}
Dividing by $2(c_i^T z)^2$ we must have:
\begin{align*}
    - c_i \cdot \frac {z^T (C^TC+\lambda I)z}{c_i^T z} = (C^TC+\lambda I)z.
\end{align*}
For this to hold we must have $(C^T C + \lambda I)z$ equal to a multiple of $c_i$ and so $z = \alpha \cdot (C^T C + \lambda I)^{-1} c_i$ for some $\alpha$. Note that the value of $\alpha$ does not change the value of $\sigma(z)$ since it simply scales the numerator and denominator in the same way. So we have that
$$
z^\star =
\argmax_{\{z \in \mathbb{R}^d: \norm{z}_2> 0\}} \frac{[Cz]_i^2}{\norm{Cz}_2^2 +\lambda \norm{z}_2^2} = (C^T C + \lambda I)^{-1} z_i.$$
Plugging in we have:
  
\begin{align*}
    \max_{\{z \in \mathbb{R}^d: \norm{z}_2> 0\}} \frac{[Cz]_i^2}{\norm{Cz}_2^2 +\lambda \norm{z}_2^2} = \frac{\left (c_i^T (C^T C + \lambda I)^{-1} c_i\right)^2}{c_i^T (C^T C + \lambda I)^{-1} (C^T C + \lambda I) (C^T C + \lambda I)^{-1} c_i} = c_i^T (C^T C+ \lambda I)^{-1} c_i,
\end{align*}
 
which completes the proof.
\end{proof}

\begin{proof}[Proof of Theorem~\ref{thm:quad}]
Letting $z = x-y$ and $\eta  = \min_{x \in B(r,y)} \tilde F_{\lambda,y}(x) $
 we can write:
\begin{align}
     \tilde F_{\lambda,y}(x) 
    & = \underbrace{\left [\frac{1}{2}\norm{H_y^{1/2} Az + H_y^{-1/2}\alpha_y}^2 + \lambda \norm{z}^2+ \gamma(z+y) \right ]}_{: G(z) = \sum_{i=1}^n g_i(z) + \lambda \norm{z}^2 + \gamma(z+y)}  +\underbrace{\left[ F(y)-\frac{1}{4}\norm{H^{-1/2}\alpha_y}^2\right]}_{\Delta = \sum_{i=1}^n {\Delta_i}}.
\end{align}
where $g_i(z) = \frac{1}{2} (H_y^{1/2}Az+H_y^{-1/2}\alpha_y)_i^2$ and 
 $\Delta_i = f_i(a_i^T y) - \frac{1}{4} \left (H_y^{-1/2}\alpha_y \right)_i^2$.
Noting that $G(z)$ is nonnegative, we can write the sensitivity as:
\begin{align}\label{eq:firstSplit}
 \sigma_{\tilde F_{\lambda,y},B(r,y)}(a_i) = \max_{\{z: \norm{z} \le r\}} \frac{g_i(z) + \Delta_i}{G(z)+\Delta} &= \max_{\{z: \norm{z} < r\}} \left [\frac{g_i(z)}{G(z)} \cdot \frac{G(z)}{G(z)+\Delta} + \frac{\Delta_i}{G(z) + \Delta}  \right ]\nonumber\\
 & \le \max_{z \in \R^d} \left [\frac{g_i(z)}{G(z)}\cdot\frac{G(z)}{G(z)+\Delta}  \right ] + \frac{f_i(a_i^T y)}{\eta}
\end{align}
since $G(z) + \Delta = \tilde F_{\lambda,y}(y+z) \ge \eta$ for $\eta  = \min_{x \in B(r,y)} \tilde F_{\lambda,y}(x) $ and since $f_i(a_i^T y) \ge \Delta_i$.
%
When $\Delta \ge 0$, $\frac{G(z)}{G(z)+\Delta} \le 1$. When $\Delta < 0$:
\begin{align*}
\frac{G(z)}{G(z)+\Delta}  = 1 - \frac{\Delta}{G(z)+\Delta} = 1 - \frac{\Delta}{\tilde F_{\lambda,y}(x)} \le 1 - \frac{\Delta}{\eta}.
\end{align*}
Overall we have:
\begin{align}\label{eq:beforeLeverageIntro}
\sigma_{\tilde F_{\lambda,y},\mathcal{W}_\eta}(a_i) &\le \max \left (1, 1 - \frac{\Delta}{\eta} \right ) \cdot \max_{\{z: z+y \in \mathcal{W}_\eta\}} \left [\frac{g_i(z)}{G(z)} \right ]  + \frac{f_i(a_i^T y)}{\eta}.
\end{align}
Letting $\displaystyle \delta = \min_{x \in B(r,y)} \gamma(x) = \min_{z:\norm{z} \le r} \gamma(z+y)$,
$C \in \R^{n \times d+1}$ be the matrix $[H_y^{1/2} A, \frac{1}{\delta}H_y^{-1/2} \alpha_y]$ and $\bar z = [z,-\delta]$ we have: 
\begin{align}\label{eq:zbar}
\frac{g_i(z)}{G(z)} = \frac{(C\bar z)_i^2}{\norm{C\bar z}^2 + \lambda \norm{z}^2 + \gamma(z+y)} = \frac{(C\bar z)_i^2}{\norm{C\bar z}^2 + \lambda \norm{\bar z}^2 - \delta + \gamma(z+y)}
\end{align}
 We can bound this ratio using Lemma \ref{lem:leverage}. Specifically, since $\gamma(z+y) - \delta \ge 0$ the ratio by $\ell_i^\lambda(C)$. Plugging back into \eqref{eq:beforeLeverageIntro} we have:

\begin{align*}
\sigma_{\tilde F_{\lambda,y},\mathcal{W}_\eta}(a_i) &\le \max \left (1, 1 - \frac{\Delta}{\eta} \right ) \cdot \ell_i^\lambda(C) + \frac{f_i(a_i^T y)}{\eta},
\end{align*}
which completes the proof.
\end{proof}

\section{Local Sensitivity Bound via Quadratic Approximation} \label{ap:true_local}
We next prove Theorem \ref{thm:local_true_sens}, which bounds the local sensitivities of a function in terms of the sensitivities of a quadratic approximation to that function, which can in term be bounded using the leverage scores of an appropriate matrix (Theorem \ref{thm:quad}).
\begin{reptheorem}{thm:local_true_sens}
Consider $F_{\lambda,y}$ as in Defs. \ref{def:fs} and \ref{def:prox}, $y \in \mathcal{X}$, radius $r$, and $\displaystyle \alpha = \min_{x \in B(r,y)} F_{\lambda,y}(x)$. We have:
  
\begin{align*}
{\sigma}_{ F_{\lambda,y},B(r,y)}(a_i) \le {\sigma}_{\tilde F_{\lambda,y},B(r,y)}(a_i) + \min \left (\frac{C_i r}{6n\lambda}, \frac{C_i r^3}{6n \alpha}  \right), ~\forall ~ i \in [n].
\end{align*}
\end{reptheorem}

\begin{proof}
From the local quadratic approximation, we have :
\begin{align*}
F_{\lambda,y}(x) = \tilde{F}_{\lambda,y}(x) + B_y(x)\| x -y \|^{3}, \text{ where }  \tilde{F}_{\lambda,y}(x) = \frac{1}{n}\sum_{i=1}^n \tilde{f}_i (a_i^\top x) +\gamma(x) +  \lambda \| x - y \|^2.
\end{align*}
From the previous Theorem~\ref{thm:quad}, we have a bound on the sensitivity for quadratic approximation,
\begin{align*}
{\sigma}_{\tilde F_{\lambda,y}, B(r,y)}(a_i)  = \sup_{x \in  {B}(r,y)}  \frac{\frac{1}{n}\tilde{f}_i (a_i^\top x)}{\tilde{F}_{\lambda,y}(x)}
\end{align*}
We can bound the local sensitivity of the true function $F_{\lambda,y}$ by:
\begin{align*}
{\sigma}_{ F_{\lambda,y}, B(r,y)}(a_i)  = \sup_{x \in{B}(r,y)}  \frac{\frac{1}{n} {f}_i (a_i^\top x)}{{F}_{\lambda,y}(x)} = \sup_{x \in{B}(r,y)}   \frac{\frac{1}{n} \left[\tilde{f}_i (a_i^\top x) + B_y^{(i)}(x)\| x -y \|^{3}\right]}{{F}_{\lambda,y}(x)}
\end{align*}
We have assumed that $B_y(x)=\frac{1}{n}\sum_{i=1}^nB_y^{(i)}(x)$ is positive
for $x \in{B}(r,y)$ and that $B_y^{(i)}(x) \leq \frac{1}{6}C_i $ for all $i$. 
This gives:
\begin{align*}
{\sigma}_{ F_{\lambda,y},B(r,y)}(a_i)  &= \sup_{x \in{B}(r,y)}  \frac{\frac{1}{n} \left[\tilde{f}_i (a_i^\top x) + B_y^{(i)}(x)\| x -y \|^{3}\right]}{{F}_{\lambda,y}(x)} \\
& \leq \underbrace{\sup_{x \in{B}(r,y)} \frac{\frac{1}{n} \left[\tilde{f}_i (a_i^\top x) \right]}{{F}_{\lambda,y}(x)}}_{\text{:= term 1}}  +\underbrace{\sup_{x \in{B}(r,y)}  \frac{C_i}{6n} \frac{\| x -y \|^{3}}{F_{\lambda,y}(x)}}_{\text{:=term 2}}.
\end{align*}
For term 1 we have:
\begin{align*}
\sup_{x \in{B}(r,y)} \frac{\frac{1}{n} \left[\tilde{f}_i (a_i^\top x) \right]}{{F}_{\lambda,y}(x)} = \sup_{x \in{B}(r,y)} \frac{\frac{1}{n} \left[\tilde{f}_i (a_i^\top x) \right]}{\tilde{F}_{\lambda,y}(x) + B_y(x) \| x - y\|^3 } \leq  {\sigma}_{\tilde F_{\lambda,y},B(r,y)}(a_i) ,
\end{align*}
where the inequality comes from assumption that $B_y(x) > 0$ for $x \in{B}(r,y)$. For term 2 we simply bound $F_{\lambda,y}(x) \ge \alpha := \min_{x \in B(r,y)} F_{\lambda,y}(x)$ or alternatively, $F_{\lambda,y}(x) \ge \lambda \norm{x-y}^2$ giving:
\begin{align*}
\frac{C_i}{6n} \frac{\| x -y \|^{3}}{F_{\lambda,y}(x)} \le \min \left (\frac{C_i r}{6n \lambda}, \frac{C_i r^3}{6n \alpha}  \right),
\end{align*}
which completes the theorem. 
\end{proof}

\section{Constrained Penalized Connection}\label{app:constPenalized}

\begin{proof}[Proof of Lemma~\ref{lem:Equivalence}]
Given,  $x^\star = \argmin_{x \in \mathbb{R}^d}     F(x) $. We assume that $F$ is a convex function.  From KKT conditions,  if $x^\star$ does not lie inside the ball than the optimal solution will exist on the boundary of the ball. Hence, the inequality in the equation can be replaced with the equality given that $x^\star $ doesn't lie inside the ball represented by the equations $ \| x - y \|^2 =r^2$. The optimization problem then becomes:
\begin{align}
 x_{r,y}^\star = \argmin_{x \in r^d}     F(x)  ~\text{such that } \| x - y \|^2 = r^2 \label{eq:ppm_equation_constrained_equ}
\end{align}
The Lagrangian of  equation~\eqref{eq:ppm_equation_constrained_equ} is: $ L(x, \nu) = F(x) +  \frac{\nu}{2} \left( \| x - y \|^2 - r^2 \right).$
First order optimality condition for the above equation implies $\nabla F(x_{r,y}^\star) +  \nu^\star (x_{r,y}^\star - y) = 0  \Rightarrow ~  x_{r,y}^\star - y = \frac{-1}{\nu^\star}  \nabla F(x_{r,y}^\star)$. 
Now from the constrained we have, $\| \frac{-1}{\nu^\star}  \nabla F(x_{r,y}^\star) \| = r ~\Rightarrow~ \nu^\star = \frac{\|\nabla F(x_{r,y}^\star) \|}{r}$. Hence, it is clear from the above arguement that $x_{r,y}^\star$ also optimize the following optimization problem:
\begin{align}
x_{r,y}^\star = \argmin_{x \in \mathbb{R}^d}  \left[ F(x) +  \frac{\|\nabla F(x_{\hat{ R},y}^\star) \|}{2r}  \| x - y \|^2  \right] \label{eq:penalized_ppm}
\end{align}
\end{proof}

\begin{proof}[Proof of Lemma~\ref{lem:radius_bound}]
As we have:
\begin{align*}
   F_{\lambda,y}(x) =  F(x) + \lambda \| x - y \|^2
\end{align*}
From the property of strongly convex function:
\begin{align}
    \|\nabla F_{\lambda,y}(y) \| = \| \nabla F(y)\| \geq  (\mu + 2\lambda) \| y - x_{\lambda,y}^\star \| \label{eq:strong_convex_ppm}
\end{align}
Now from the first order optimality of $F_{\lambda,y}$, we have:
\begin{align*}
    \nabla F_{\lambda,y} (x_{\lambda,y}^\star) = \nabla F(x_{\lambda,y}^\star) + 2{\lambda}( x_{\lambda,y}^\star - y)  =0
\end{align*}
Hence,
\begin{align}
    \|\nabla F(x_{\lambda,y}^\star) \| = 2\lambda \| x_{\lambda,y}^\star - y \| \label{eq:strong_convex_ppm_opt}
\end{align}
From the equations~\eqref{eq:strong_convex_ppm} and \eqref{eq:strong_convex_ppm_opt}, we have:
\begin{align*}
    \|\nabla F(x_{\lambda,y}^\star) \| \leq \frac{2\lambda}{\mu +2 \lambda} \|\nabla F(y) \|
\end{align*}
From the equation~\eqref{eq:penalized_ppm}, we know that 
\begin{align*}
    R =  \frac{\| \nabla F(x_{\lambda,y}^\star) \|}{2\lambda } \leq \frac{\| \nabla F(y) \|}{2\lambda + \mu }
\end{align*}
If the optimal point $x^\star$ of the function $F$ lie in the ball then the radius will be further less.  
\end{proof}

\begin{corollary}  \label{cor:radius_bound_ap}
After running one step of line~4 of the Algorithm~\ref{algo:approximate_ppm} for the parameters $x_{t-1}$, $\lambda_t$, $\epsilon_t$ and $\mu$, we have the following bound:
\begin{align*}
\| x_t - x_{t-1} \| &\leq \sqrt{\frac{2 \epsilon_{t}}{2\lambda_t + \mu}  F(x_{t -1})} + \frac{\| \nabla F(x_{t-1}) \|}{2\lambda_t + \mu}  \\
\| x_t - x_{t-1} \| &\geq r_t^\star - \sqrt{\frac{2 \epsilon_{t}}{2\lambda_t + \mu}  F(x_{t -1})} 
\end{align*}
where $r_t^\star = \| x_{t-1} - x^\star_{\lambda_t} \|$.
\end{corollary}
\begin{proof}

As from Lemma~\ref{lem:radius_bound}, we have 
\begin{align*}
\| x_{2\lambda_t, x_{t-1}}^\star  - x_{t-1} \|  \leq \frac{\| \nabla F(x_{t-1}) \|}{2 \lambda_t + \mu}.
\end{align*}

Let us denote $\| x_{\lambda_t, x_{t-1}}^\star  - x_{t-1} \|$ as $r_t$. Now, let us try to bound $\| x_{t} -  x_{\lambda_t, x_{t-1}}^\star \|$.
From the strong convexity and aproximation argument:
\begin{align*}
\| x_{t} -  x_{\lambda_t, x_{t-1}}^\star \|^2 \leq \frac{2}{2\lambda_t + \mu} \left(  F_{\lambda_t ,x_{t-1}}(x_t) - f^\star_{\lambda_t, x_{t-1}}\right) \l \leq  \frac{2 \epsilon_{t}}{2\lambda_t + \mu} f^\star_{\lambda_t, x_{t-1}}
\end{align*}
Now we can apply strong convexity argument one more time. 
$$f^\star_{\lambda_t, x_{t-1}} \leq F(x_{t -1}) - \frac{2\lambda_t + \mu}{2} r_t^2$$
Hence finally we have:
\begin{align}
\| x_{t} -  x_{\lambda_t, x_{t-1}}^\star \|^2 \leq \frac{2 \epsilon_{t}}{2\lambda_t + \mu}  F(x_{t -1})  - \epsilon_t r_t^2 \label{eq:radius_ap_bound}
\end{align}
Hence finally:
\begin{align*}
 r_t - \sqrt{\frac{2 \epsilon_{t}}{2\lambda_t + \mu}  F(x_{t -1})}  \leq  \| x_t - x_{t-1} \| \leq \sqrt{\frac{2 \epsilon_{t}}{2\lambda_t + \mu}  F(x_{t -1})} + r_t \leq \sqrt{\frac{2 \epsilon_{t}}{2\lambda_t + \mu}  F(x_{t -1})} + \frac{\| \nabla F(x_{t-1}) \|}{2\lambda_t + \mu}
\end{align*}
\end{proof}

\section{Approximate Proximal Point Method } \label{ap:app_ppm}

The following Lemma from~\ref{lem:frostig} is useful in proving the Theorem~\ref{thm:approximate_ppm_coreset}.
\begin{lemma}[Lemma 2.7 \cite{frostig2015regularizing}] \label{lem:frostig}
For all $y \in \mathbb{R}^d$ and $\lambda \geq 0$:
$$F(x^\star_{\lambda,y}) - F^\star \le F_{\lambda,y}^\star - F^\star \leq \frac{2\lambda}{\mu + 2\lambda} \left( F(y) - F^\star \right). $$
\end{lemma}

\begin{proof}[Proof of Theorem~\ref{thm:approximate_ppm_coreset}]
Let us assume that $x_{\lambda,x}^\star  = \argmin _{y \in \mathbb{R}^d} F_{\lambda,x}(y)$, then from the Lemma~\ref{lem:frostig}, 
\begin{align}\label{eq:needThisOne}
\begin{split}
    &F_{\lambda,x}^\star - F^\star \leq \frac{2\lambda}{\mu + 2\lambda} \left( F(x) - F^\star \right)  \\
   \Rightarrow ~~& F(x_{\lambda,x}^\star) - F^\star \leq \frac{2\lambda}{\mu + 2\lambda} \left( F(x) - F^\star \right)
   \end{split}
\end{align}
Last equation comes from the fact that $F_{\lambda,x}^\star = F(x_{\lambda,x}^\star) + {\lambda} \|  x^\star -x \|^2$. 

We know that $$F_{\lambda,y}(x_{\lambda,y}^\star) \leq f\big(\mathcal{P}_{F_{\lambda,y}}(x)\big) \leq (1+\epsilon) F_{\lambda,y}(x_{\lambda,y}^\star)~\forall ~y \in \mathbb{R}^d.$$
We can get the upper bound on the true minimizer using this black-box oracle in terms of the approximate solution. We have:
\begin{align}
    F_{\lambda_T,x_{t-1}}^\star\leq  F_{\lambda_T,x_{t-1}} (x_t) \label{eq:upper_bnd_appro_true}
\end{align}
From Lemma~\ref{lem:frostig} and black-box oracle ,  for any $t \in [T]$ we have
\begin{align}
\begin{split} \label{eq:approximate_progress}
    F_{\lambda_t,x_{t-1}}(x_t) - F^\star &= F_{\lambda_t,x_{t-1}}(x_t)  - F_{\lambda_t,x_{t-1}}^\star + F_{\lambda_t,x_{t-1}}^\star - F^\star \\
    & \leq \epsilon_t ~F_{\lambda,x_{t-1}}^\star + \frac{2\lambda_t}{\mu + 2\lambda_t} \left( F(x_{t-1}) - F^\star \right) \\
     & \leq {\epsilon_t} ~F_{\lambda_t,x_{t-1}}(x_t) + \frac{2\lambda_t}{\mu + 2\lambda_t} \left( F(x_{t-1}) - F^\star \right)  
\end{split}
\end{align}
which leads us to 
\begin{align}
    \begin{split}
       &(1 - \epsilon_t) F_{\lambda_t,x_{t-1}}(x_t)  - F^\star \leq \frac{2\lambda_t}{\mu + 2\lambda_t} \left( F(x_{t-1}) - F^\star \right)  \\
       \Rightarrow~&  (1 - \epsilon_t) F_{\lambda_t,x_{t-1}}(x_t)  - (1 - \epsilon_t) F^\star \leq \frac{2\lambda_t}{\mu + 2\lambda_t} \left( F(x_{t-1}) - F^\star \right) + {\epsilon_t} F^\star \\
       \Rightarrow~&  F_{\lambda_t,x_{t-1}}(x_t)    - F^\star \leq \frac{1 } {1 -  \epsilon_t}\frac{2\lambda_t}{\mu + 2\lambda_t} \left( F(x_{t-1}) - F^\star \right) + \frac{  \epsilon_t} {1 -  \epsilon_t} F^\star
    \end{split}
\end{align}
Now since,  $$ F_{\lambda_t,x_{t-1}}(x_t) = F(x_t) + {\lambda_t} \| x_t - x_{t-1} \|^2 \geq F(x_t) $$
Hence, finally we have:
\begin{align}
    F(x_t)    - F^\star &\leq \frac{1 } {1 -  \epsilon_t}\frac{2\lambda_t}{\mu + 2\lambda_t} \left( F(x_{t-1}) - F^\star \right) + \frac{  \epsilon_t} {1 -  \epsilon_t} F^\star \notag \\
    &\leq  (1+ 2\epsilon_t) \frac{2\lambda_t}{\mu + 2\lambda_t} \left( F(x_{t-1}) - F^\star \right)  + 2\epsilon_t F^\star  \label{eq:recursion_1}   
\end{align}
whenever $\epsilon_t \leq 1/2$.
Now we can do recursion on the equation~\eqref{eq:recursion_1}:
\begin{align}
    F(x_T)    - F^\star \leq \underbrace{\Bigg[\prod_{t=1}^T  (1+ 2\epsilon_t) \frac{2\lambda_t}{\mu+2\lambda_t} \Bigg] }_{:\text{linear rate}}\left(F(x_0) - F^\star \right) + F^\star \underbrace{ \left[ \mathlarger{\mathlarger{\sum}}_{t=1}^T {2\epsilon_t} \prod_{j = t+1}^T{(1 + 2 \epsilon_j)} \frac{2\lambda_j}{\mu + 2\lambda_j}  \right]}_{:= \delta} \label{eq:last_thm11}
\end{align}

\end{proof}

{
\begin{algorithm}[H]
\begin{algorithmic}[1] 
  \STATE  \textbf{input} $x_0 \in \mathbb{R}^d$, $\lambda_t > 0 ~\forall~ {t\in [T]}$.
  \STATE \textbf{for}~ {$t=1\dots \text{ T}$} \textbf{do}
  \STATE \quad  $x^\star_{\lambda_t, x_{t-1}} \gets \argmin F(x ) + {\lambda_t} \|x - x_{t-1} \|^2 $ \label{alg_lin:exact_g_lin3_smooth}
  \STATE \quad  $x_t \gets x^\star_{\lambda_t, x_{t-1}} $
    \STATE \textbf{end for}
  \STATE \textbf{output} $x_T$ 
\end{algorithmic}
 \caption{Proximal-Point Method}
 \label{algo:exact_ppm}
\end{algorithm}
}
\begin{lemma}[Proposition 3.1.6 \cite{tichatschkeproximal}] \label{lem:smooth_exact_ppm}
Let $F$ be lower semi-continuous convex function then for any $x$ in the domain and for any $t \geq 1$ following relation holds for iterates in Algorithm~\ref{algo:exact_ppm}:
\begin{align*}
\frac{1}{\lambda_{t}} \left(F(x) - F(x^\star_{\lambda_t, x_{t-1}}) \right) \geq \| x_{t-1} - x^\star_{\lambda_t, x_{t-1}} \|^2 + \| x  - x^\star_{\lambda_t, x_{t-1}}\|^2 - \| x - x_{t-1} \|^2.
\end{align*}

\end{lemma}
In the next lemma, we characterize the result provided in lemma~\ref{lem:smooth_exact_ppm} for the $\epsilon$-approximate oracle. 

 \begin{lemma} \label{lem:smooth_appt_ppm}
 Let $F$ be lower semi-continuous convex function then for $x^\star$, the minimizer of $F$ and for any $t \geq 1$ and $\epsilon \leq 1/2$, following relation holds for iterates in Algorithm~\ref{algo:approximate_ppm}:
\begin{align*}
\frac{1 }{\lambda_{t}}  \left(F_{\lambda_t,x_{t-1}}(x_t) - F^\star  \right) &\leq \frac{2}{\lambda_{t}} \left(F_{\lambda_t,x_{t-1}}(x^\star_{\lambda_t, x_{t-1}}) - F^\star \right) + \frac{\epsilon_t}{(1 - \epsilon_t)\lambda_t}  F^\star  \\
&\leq 2\| x^\star - x_{t-1} \|^2 - 2\| x^\star  -  x^\star_{\lambda_t, x_{t-1}}\|^2    + \frac{\epsilon_t}{(1 - \epsilon_t)\lambda_t}  F^\star
\end{align*}
\end{lemma}
\begin{proof}
We have $x_t = {P}_{f_{\lambda_t,x_{t-1}}}(x)$ as defined in line~3 of Algorithm~\ref{algo:approximate_ppm} where  $\mathcal{P}_f$ is \textit{multiplicative }$\epsilon_t$-oracle. \\ 
From the oracle we know that $F_{\lambda,x_{t-1}}(x^\star_{\lambda_t, x_{t-1}}) \leq F_{\lambda,x_{t-1}}(x_t) \leq (1 +\epsilon_t) F_{\lambda,x_{t-1}}(x^\star_{\lambda_t, x_{t-1}}) $. Next we use the result from Lemma~\ref{lem:smooth_exact_ppm} where we use $x  = x^\star = \argmin_{x} F(x)$. We denote $F^\star$ with $F(x^\star)$.
\begin{align}
\begin{split} \label{eq:smooth_appm}
\frac{1}{\lambda_{t}} \left(F^\star - F(x^\star_{\lambda_t, x_{t-1}}) \right) &\geq \| x_{t-1} - x^\star_{\lambda_t, x_{t-1}} \|^2 + \| x^\star  - x^\star_{\lambda_t, x_{t-1}}\|^2 - \| x^\star - x_{t-1} \|^2 \\
&\geq \| x_{t-1} - x^\star_{\lambda_t, x_{t-1}} \|^2 + \| x^\star  - x_t + x_t - x^\star_{\lambda_t, x_{t-1}}\|^2 - \| x^\star - x_{t-1} \|^2 \\
\end{split}
\end{align}

The last equation essentially tells us the following:
\begin{align}
\frac{1}{\lambda_{t}} \left(F^\star - F_{\lambda_t,x_{t-1}}(x^\star_{\lambda_t, x_{t-1}}) \right) &\geq \| x^\star  -  x^\star_{\lambda_t, x_{t-1}}\|^2 - \| x^\star - x_{t-1} \|^2 \label{eq:smooth_appm_quad}
\end{align}

From the $\epsilon_t$-oracle we do have:
\begin{align*}
\left(  F_{\lambda_t ,x_{t-1}}(x_t) - F_{\lambda,x_{t-1}}(x^\star_{\lambda_t, x_{t-1}})\right) \leq \epsilon_{t} F_{\lambda,x_{t-1}}(x^\star_{\lambda_t, x_{t-1}}) \\
\end{align*}
Hence
\begin{align}
\begin{split} \label{eq:smooth_approx}
\frac{1}{\lambda_{t}} \left(F_{\lambda_t,x_{t-1}}(x_t) - F^\star  \right) &= \frac{1}{\lambda_{t}} \left(F_{\lambda_t,x_{t-1}}(x_t) - F_{\lambda_t,x_{t-1}}(x^\star_{\lambda_t, x_{t-1}}) + F_{\lambda_t,x_{t-1}}(x^\star_{\lambda_t, x_{t-1}}) - F^\star  \right) \\
& = \frac{1}{\lambda_{t}} \Big[ \left(F_{\lambda_t,x_{t-1}}(x_t) - F_{\lambda_t,x_{t-1}}(x^\star_{\lambda_t, x_{t-1}})\right) + \left( F_{\lambda_t,x_{t-1}}(x^\star_{\lambda_t, x_{t-1}}) - F^\star \right) \Big] \\
& \leq \frac{1}{\lambda_{t}}  \Big[ \epsilon_t F_{\lambda_t,x_{t-1}}(x^\star_{\lambda_t, x_{t-1}}) + \left( F_{\lambda_t,x_{t-1}}(x^\star_{\lambda_t, x_{t-1}}) - F^\star \right)   \Big] \\
& \leq \frac{1}{\lambda_{t}}  \Big[ \epsilon_t F_{\lambda_t,x_{t-1}}(x_t) + \left( F_{\lambda_t,x_{t-1}}(x^\star_{\lambda_t, x_{t-1}}) - F^\star \right)   \Big] \\
\end{split}
\end{align}

From equations~\eqref{eq:smooth_appm_quad} and \eqref{eq:smooth_approx}, we have:
\begin{align*}
&\frac{1}{\lambda_{t}} \left((1 - \epsilon_t) F_{\lambda_t,x_{t-1}}(x_t) - F^\star  \right) \leq \frac{1}{\lambda_{t}} \left(F_{\lambda_t,x_{t-1}}(x^\star_{\lambda_t, x_{t-1}}) - F^\star \right)  \\
\Rightarrow~&\frac{1 (1 - \epsilon_t)}{\lambda_{t}}  \left(F_{\lambda_t,x_{t-1}}(x_t) - F^\star  \right) \leq \frac{1}{\lambda_{t}} \left(F_{\lambda_t,x_{t-1}}(x^\star_{\lambda_t, x_{t-1}}) - F^\star \right) + \frac{ \epsilon_t}{\lambda_t} F^\star \\ 
\Rightarrow~&\frac{1 }{\lambda_{t}}  \left(F_{\lambda_t,x_{t-1}}(x_t) - F^\star  \right) \leq \frac{1}{(1 - \epsilon_t) \lambda_{t}} \left(F_{\lambda_t,x_{t-1}}(x^\star_{\lambda_t, x_{t-1}}) - F^\star \right) + \frac{\epsilon_t}{(1 - \epsilon_t)\lambda_t} F^\star 
\end{align*}
If $\epsilon_t \leq 1/2 $, then from equations~\eqref{eq:smooth_appm_quad} and \eqref{eq:smooth_approx}, we have:
\begin{align*}
\frac{1 }{\lambda_{t}}  \left(F_{\lambda_t,x_{t-1}}(x_t) - F^\star  \right) &\leq \frac{2}{\lambda_{t}} \left(F_{\lambda_t,x_{t-1}}(x^\star_{\lambda_t, x_{t-1}}) - F^\star \right) + \frac{ \epsilon_t}{(1 - \epsilon_t)\lambda_t}  F^\star  \\
&\leq 2\| x^\star - x_{t-1} \|^2 - 2\| x^\star  -  x^\star_{\lambda_t, x_{t-1}}\|^2    + \frac{ \epsilon_t}{(1 - \epsilon_t)\lambda_t}  F^\star
\end{align*}

\end{proof}

\begin{lemma} \label{lem:almost_final_smooth}
For a lower semi-continuous convex function $F$ at any and  for any $t \geq 1$ and $\epsilon \leq 1/2$, following relation holds for iterates after $T$ iterations in Algorithm~\ref{algo:approximate_ppm}:
\begin{align*}
\sum_{t=1}^T\frac{1 }{\lambda_{t}}  \left(F_{\lambda_t,x_{t-1}}(x_t) - F^\star  \right) \leq \frac{2}{(1-\epsilon)}\| x^\star - x_{0} \|^2 + \sum_{t=1}^T \frac{3\epsilon}{((1-\epsilon))\lambda_t} ~F^\star 
\end{align*}
\end{lemma}

\begin{proof}
We know that:
\begin{align}
\frac{1}{\lambda_{t}} \left(F_{\lambda_t,x_{t-1}}(x^\star_{\lambda_t, x_{t-1}}) - F^\star \right) \leq   \| x^\star - x_{t-1} \|^2 - \| x^\star  -  x^\star_{\lambda_t, x_{t-1}}\|^2 \label{eq:eq_basic_smooth}
\end{align}
We can however sum the equation~\eqref{eq:eq_basic_smooth} for $t=1$ till T and we get:
\begin{align}
\begin{split}  \label{eq:recur_sum_smooth}
\sum_{t=1}^T\frac{1}{\lambda_{t}} \left(F_{\lambda_t,x_{t-1}}(x^\star_{\lambda_t, x_{t-1}}) - F^\star \right)& \leq  \sum_{t=1}^T \Big[ \| x^\star - x_{t-1} \|^2 - \| x^\star  -  x^\star_{\lambda_t, x_{t-1}}\|^2 \Big] \\
&= \| x^\star - x_{0} \|^2 + \sum_{t=1}^{T-1} \Big[ \| x^\star - x_{t} \|^2 - \| x^\star  -  x^\star_{\lambda_t, x_{t-1}}\|^2\Big]  \\ 
& \qquad \qquad  \qquad \qquad  \qquad \qquad  \qquad  - \| x^\star  -  x^\star_{\lambda_t, x_{T-1}}\|^2 \\
&\leq  \| x^\star - x_{0} \|^2  + \sum_{t=1}^{T} \Big[ \| x^\star - x_{t} \|^2 - \| x^\star  -  x^\star_{\lambda_t, x_{t-1}}\|^2\Big]  
\end{split}
\end{align}

In equation~\eqref{eq:recur_sum_smooth}, we can use Corollary~\ref{cor:radius_bound_ap},
\begin{align*}
\| x^\star - x_{t} \|^2 - \| x^\star  -  x^\star_{\lambda_t, x_{t-1}}\|^2 \leq  \frac{ \epsilon_{t}}{\lambda_t }  F_{\lambda_t,x_{t-1}}(x^\star_{\lambda_t, x_{t-1}}).
\end{align*}
Hence,
\begin{align*}
&\sum_{t=1}^T\frac{1}{\lambda_{t}} \left(F_{\lambda_t,x_{t-1}}(x^\star_{\lambda_t, x_{t-1}}) - F^\star \right) \leq  \| x^\star - x_{0} \|^2   + \sum_{t=1}^{T} \frac{ \epsilon_{t}}{\lambda_t }  F_{\lambda_t,x_{t-1}}(x^\star_{\lambda_t, x_{t-1}}) \\
\Rightarrow~& \sum_{t=1}^T\frac{1}{\lambda_{t}} \left((1-\epsilon_t)F_{\lambda_t,x_{t-1}}(x^\star_{\lambda_t, x_{t-1}}) - F^\star \right) \leq  \| x^\star - x_{0} \|^2  \\
\Rightarrow~& \sum_{t=1}^T\frac{(1-\epsilon_t)}{\lambda_{t}} \left(F_{\lambda_t,x_{t-1}}(x^\star_{\lambda_t, x_{t-1}}) - F^\star \right) \leq  \| x^\star - x_{0} \|^2 + \sum_{t=1}^T \frac{ \epsilon_t}{\lambda_t} ~F^\star 
\end{align*}

Now, if we choose $\epsilon_t = \epsilon$ for all $t$ then we have:
\begin{align}
\sum_{t=1}^T\frac{1}{\lambda_{t}} \left(F_{\lambda_t,x_{t-1}}(x^\star_{\lambda_t, x_{t-1}}) - F^\star \right) \leq  \frac{1}{(1-\epsilon)}\| x^\star - x_{0} \|^2 + \sum_{t=1}^T \frac{ \epsilon}{((1-\epsilon))\lambda_t} ~F^\star  \label{eq:error_Smooth_F_lambda}
\end{align}

From the previous lemma~\ref{lem:smooth_appt_ppm}, we have 
\begin{align}
\frac{1 }{\lambda_{t}}  \left(F_{\lambda_t,x_{t-1}}(x_t) - F^\star  \right) &\leq \frac{2}{\lambda_{t}} \left(F_{\lambda_t,x_{t-1}}(x^\star_{\lambda_t, x_{t-1}}) - F^\star \right) + \frac{\epsilon_t}{(1 - \epsilon_t)\lambda_t}  F^\star \label{eq:prev_lemma}
\end{align}

Summing up the equation~\eqref{eq:prev_lemma} for $t=1$ to $T$ and for $\epsilon_t = \epsilon$, we have:
\begin{align}
\sum_{t=1}^T \frac{1 }{\lambda_{t}}  \left(F_{\lambda_t,x_{t-1}}(x_t) - F^\star  \right) &\leq \sum_{t=1}^T \frac{2}{\lambda_{t}} \left(F_{\lambda_t,x_{t-1}}(x^\star_{\lambda_t, x_{t-1}}) - F^\star \right) + \sum_{t=1}^T \frac{ \epsilon}{(1 - \epsilon)\lambda_t} ~ F^\star  \label{eq:penultimate_smooth_appm}
\end{align}

Now from equations~\eqref{eq:error_Smooth_F_lambda} and \eqref{eq:penultimate_smooth_appm}, 
\begin{align}
\sum_{t=1}^T \frac{1 }{\lambda_{t}}  \left(F_{\lambda_t,x_{t-1}}(x_t) - F^\star  \right) &\leq \sum_{t=1}^T \frac{2}{\lambda_{t}} \left(F_{\lambda_t,x_{t-1}}(x^\star_{\lambda_t, x_{t-1}}) - F^\star \right) + \sum_{t=1}^T \frac{ \epsilon}{(1 - \epsilon)\lambda_t} ~ F^\star  \\
&\leq  \frac{2}{(1-\epsilon)}\| x^\star - x_{0} \|^2 + \sum_{t=1}^T \frac{3\epsilon}{((1-\epsilon))\lambda_t} ~F^\star 
\end{align}

\end{proof}

\begin{proof}[Proof of Theorem~\ref{thm:approximate_ppm_coreset_smooth}]

From the previous Lemma~\ref{lem:almost_final_smooth}, we have:
\begin{align*}
\sum_{t=1}^T\frac{1 }{\lambda_{t}}  \left(F_{\lambda_t,x_{t-1}}(x_t) - F^\star  \right) \leq \frac{2}{(1-\epsilon)}\| x^\star - x_{0} \|^2 + \sum_{t=1}^T \frac{3\epsilon}{((1-\epsilon))\lambda_t} ~F^\star 
\end{align*}
We assume that $F(x_t) \leq F(x_{t-1})$ for all $t$. This is fine to assume as we can always do the resampling if failed once. And also:
\begin{align*}
\frac{1 }{\lambda_{t}}  \left(F(x_t) - F^\star  \right) \leq \frac{1 }{\lambda_{t}}  \left(F_{\lambda_t,x_{t-1}}(x_t) - F^\star  \right)~~\text{for all } t.
\end{align*}
Hence, 
\begin{align*}
F(x_T) - F^\star  \leq \frac{2}{(1-\epsilon)} \frac{\| x^\star - x_{0} \|^2}{\sum_{t=1}^T \frac{1}{\lambda_t}} + \frac{3\epsilon}{1-\epsilon} F^\star.
\end{align*}

\end{proof}

\section{Adaptive Stochastic Trust Region Method} \label{ap:quad_approx}

\begin{algorithm}[H]
\begin{algorithmic}[1] 
  \STATE  \textbf{input} $x_0 \in \mathbb{R}^d$, $\epsilon_0$, $\mu$ and $m > 0$.
  \STATE Compute $\|\nabla F(x_0)\|$, $F(x_0)$ and $C_0$
  \STATE \textbf{for}~ {$t=1\dots \text{ T}$} \textbf{do}
  \STATE \quad Compute regularizer $\lambda_t$ using $\|\nabla f(x_{t-1})\|$, $\lambda_t$ and $\mu$. \label{alg_lin:exact_g_lin1}
  \STATE \quad Compute radius $r_t$ using $\|\nabla f(x_{t-1})\|$, $f(x_{t-1})$ and $C_{t-1}$. \label{alg_lin:exact_g_lin1}
  \STATE \quad Computer error parameter $\epsilon_t$ using $\lambda_t$ and $\mu$, the strong convexity of $F$.\label{alg_lin:exact_g_lin2}
  \STATE \quad Get $\tilde F_{\lambda_{t}, x_{t-1}}$ via Taylor Expansion.\label{alg_lin:exact_g_lin3}
  \STATE \quad Compute the sensitivity for $\tilde F_{\lambda_{t}, x_{t-1}}$ using Theorem \ref{thm:quad}.\label{alg_lin:exact_g_lin4}
  \STATE \quad Local sensitivity based sampling of $\tilde{F}_{\lambda_t,x_{t-1}}^s(x)$ from $\tilde{F}_{\lambda_t,x_{t-1}}(x)$.\label{alg_lin:exact_g_lin4}
   \STATE \quad $x_{t} \gets \argmin_{x \in B(r_t,x_{t-1})} F_{\lambda_t,x_{t-1}}^s(x)$.\label{alg_lin:exact_g_lin5}
    \STATE \quad  Compute $\|\nabla F(x_t)\|$, $F(x_t)$ and $C_t$.\label{alg_lin:exact_g_lin8}
   
    \STATE \textbf{end for}
  \STATE \textbf{output} $x_T$ 
\end{algorithmic}
 \caption{Adaptive Stochastic Trust Region Method}
 \label{algo:ada_trust_region}
\end{algorithm}

We here now provide the detailed statemnt of Theorem~\ref{thm:main} and then provide the proof for it. 
\begin{reptheorem}{thm:local_true_sens} 
For a given set of constants $C_k $, $\delta_k$ and $\tilde \epsilon_k  = \delta_k \frac{\mu}{\lambda_k + \mu}$ which is error tolerance for the square approximation of the function $F_{\lambda_{k},x_{k-1}}(x)$  for all $k \in [T]$, if $\lambda_{k+1}$ is chosen as :
$$ 2\lambda_{k+1} = \max\left( \sqrt{ \frac{ 4C_k\| \nabla F(x_k)\|^3}{\frac{1}{4c^2}  {\| \nabla F(x_k) \|^2} +  4 \tilde \delta_{k+1} \mu   \frac{F(x_k)}{3m} }   }- \mu , \mu \right),$$
then with probability $\geq 1/2$ the following holds:
\begin{align}
F(x_{k+1})  - F^\star \leq (1 + 2 \epsilon_{k+1}) \frac{2\lambda_{k+1}}{2\lambda_{k+1}+\mu} \left(F(x_k) - F^\star \right) + 2 \epsilon_{k+1} F^\star, \label{eq:main_thm_eq1}
\end{align}
where $\epsilon_{k+1} = 2\tilde{\epsilon}_{k+1}\left( 1  + \frac{1}{m} \right)~\forall~k$, $m$ and $c$ are positive constants. 

\end{reptheorem}

\begin{proof}[Proof of Theorem~\ref{thm:main}]
Let us first reiterate the notations:
\begin{align*}
\tilde{F}_{\lambda_{k+1},x_k}(x) = f(x_k) + (x - x_k)^\top A^\top \alpha_{x_k}+ \frac{\gamma}{2} \| x \|^2 + (x - x_k)^\top  A^\top H_{x_k} A (x - x_k)+ {\lambda}\| x - x_k \|^2. 
\end{align*}
 and $F_{\lambda_{k+1},x_k}(x) = \tilde{F}_{\lambda_{k+1},x_k} (x)+ {B_{x_{k}}} (x)\|x - x_k \|^3$. We can write $\tilde{F}_{\lambda_{k+1},x_k}(x)  = \frac{1}{n} \sum_{i=1}^n \tilde f_i(x^T a_i)$ where $ \tilde f_i(x^T a_i)$ is the quadratic approximation of $f_i(x^T a_i)$ around the point $x_k$.
We also define the upper bound on the radius $r_{k+1} = \frac{\| \nabla F(x_k)\|}{2\lambda_{k+1}+ \mu}$. Contribution in $B_{x_k}$ comes from each term $f_i$ \textit{i.e.} $B_{x_k}(x) = \frac{1}{n} \sum_{i=1}^n B_{x_k}^{(i)}(x)$. Let us assume that $x_{k+1}$ is the point, we get after minimizing the subset after sampling from the sensitivity of the quadratic approximation.  To make proof simpler in this section,  we assume $C_k^{(i)}$ as the upper bound on the absolute value of $B_{x_k}^{(i)} (x) ~ \forall ~ i \in [n]$ in the ball $B(x_k,r_{k+1})$ \textit{i.e.}   $C_{k}^{(i)} \cdot r_{k+1}^3 \geq \max_{x \in B(x_k,r_{k+1})} \left | f_i(x^T a_i) - \tilde f_i(x^T a_i) \right | ~ \forall ~ i \in [n]$ where $C_{k}^{(i)}$ is a positive real number. We have $C_k = \frac{1}{n} \sum_{i=1}^n C_{k}^{(i)}$. \\
As we have already defined for all $x$:
\begin{align*}
\left | \tilde{F}_{\lambda_{k+1},x_k}(x) - {F}_{\lambda_{k+1},x_k}(x)\right | \leq C_k  \| x - x_{k} \|^3 .
\end{align*}

So if $ \tilde{F}^s_{\lambda_{k+1},x_k}(x)$ is sampled by sensitivities with error parameter $\tilde{\epsilon}_{k+1}$ we have by triangle inequality:

\begin{align}\label{eq:FsAdditiveApprox}
\left | \tilde{F}^s_{\lambda_{k+1},x_k}(x) - {F}_{\lambda_{k+1},x_k}(x)\right | &\le \left | \tilde{F}^s_{\lambda_{k+1},x_k}(x) - \tilde {F}_{\lambda_{k+1},x_k}(x)\right | + \left | \tilde{F}_{\lambda_{k+1},x_k}(x) - {F}_{\lambda_{k+1},x_k}(x)\right |\nonumber\\
&\le  C_k  \| x - x_{k} \|^3  + \tilde{\epsilon}_{k+1} \tilde{F}_{\lambda_{k+1},x_k}(x)\nonumber\\
&\le C_k  \| x - x_{k} \|^3  + \frac{\tilde{\epsilon}_{k+1}}{1-\tilde \epsilon_{k+1}} \cdot \tilde{F}^s_{\lambda_{k+1},x_k}(x).
\end{align}

Hence, with very high probability, we do have :
\begin{align}
\begin{split} \label{eq:FsAdditiveApprox_2}
{F}_{\lambda_{k+1},x_k}(x) &\leq C_k  \| x - x_{k} \|^3  + \frac{\tilde{\epsilon}_{k+1}}{1-\tilde \epsilon_{k+1}} \cdot \tilde{F}^s_{\lambda_{k+1},x_k}(x) + \tilde{F}^s_{\lambda_{k+1},x_k}(x) \\
&= C_k  \| x - x_{k} \|^3 + \frac{1}{1-\tilde \epsilon_{k+1}} \tilde{F}^s_{\lambda_{k+1},x_k}(x)
\end{split}
\end{align}

Now, we would like to show that letting $x_k^s = \argmin_{x \in B(x_k,r_{k+1})} \tilde{F}_{\lambda_{k+1},x_k}^s(x)$,  the error can still be controlled. 

If, we let $ x_k^s$ be the minimizer of $\tilde{F}^s_{\lambda_{k+1},x_k}(x)$ and $x_{\lambda_{k+1},x_k}^\star$ be the minimizer of ${F}_{\lambda_{k+1},x_k}(x)$. We assume that $\frac{r_{k+1}}{c} \leq \| \tilde x_k^s - x_{k} \| \le r_{k+1}$ and $\| x^\star_{\lambda_{k+1},x_k} - x_{k} \| \le r_{k+1}$ for some positive real constant $c > 1$. We  have : 
  
\begin{align}
{F}_{\lambda_{k+1},x_k}( x_k^s) &\le \frac{1}{1-\tilde \epsilon_{k+1}} \tilde{F}^s_{\lambda_{k+1},x_k}(x) + {C_k}\|  x_k^s - x_{k} \|^3   \nonumber\\
&\le \frac{1}{1-\tilde \epsilon_{k+1}}  {F}^s_{\lambda_{k+1},x_k}(x_{\lambda_{k+1},x_k}^\star)  + {C_k}\|  x_k^s - x_{k} \|^3  \label{eq:camFs1}
\end{align}
where the second line follows the fact that $ x_k^s$ minimizes $\tilde{F}^s_{\lambda_{k+1},x_k}(x)$. 

Hence, if we set $\tilde{\epsilon}_{k+1} \le 1/2$ plugging back everything together:
\begin{align}
{F}_{\lambda_{k+1},x_k}( x_k^s) &\le (1+4 \tilde \epsilon_{k+1}) F_{\lambda_{k+1},x_k}^\star + 4 C_k r_{k+1}^3\label{eq:fsClean}.
\end{align}
where in the last line we use that both $\| \tilde x_k^s - x_{k} \| \le r_{k+1}$ and $\| x^\star_{\lambda_{k+1},x_k} - x_{k} \| \le r_{k+1}$.


We have from Lemma~\ref{lem:frostig} that:
\begin{align*}
F_{\lambda_{k+1},x_k}^\star \leq \frac{2\lambda_{k+1}}{\mu + 2\lambda_{k+1}} \left( F(x_k) - F^\star \right) + F^\star.
\end{align*}
Plugging this bound into \eqref{eq:fsClean} gives:
\begin{align}\label{tildeSBound}
{F}_{\lambda_{k+1},x_k}( x_k^s)  \le  \frac{(1+4\tilde \epsilon_{k+1})2\lambda_{k+1}}{\mu +2 \lambda_{k+1}} \left( F(x_k) - F^\star \right) + F^\star + 4C_k r_{k+1}^3.
\end{align}

Now consider if we make the update $x_k^s = x_{k+1}$. Then we have using the simple bound that $F(x) \le {F}_{\lambda_{k+1},x_k}(x)$ for all $x$:
\begin{align*}
{F}( x_{k+1}) &=  {F}_{\lambda_{k+1},x_k}(x_{k+1}) -{\lambda_{k+1}} \| x_{k+1} - x_k \|^2 \\ 
\Rightarrow ~{F}( x_{k+1}) &\le \frac{(1+4\tilde \epsilon_{k+1})2\lambda_{k+1}}{\mu + 2\lambda_{k+1}} \left( F(x_k) - F^\star \right) + F^\star + 4 C_k r_{k+1}^3 - {\lambda_{k+1}} \| x_{k+1} - x_k \|^2\\
&\le \frac{(1+4\tilde \epsilon_{k+1})2\lambda_{k+1}}{\mu + 2\lambda_{k+1}} \left( F(x_k) - F^\star \right) + F^\star + 4 C_k r_{k+1}^3 - \frac{\lambda_{k+1}}{c^2} r_{k+1}^2\\
\end{align*}
In the last line we have used $ \|  x_{k+1} - x_k \| \geq \frac{r_{k+1}}{c}$.
Now, we do want to choose our parameters such that the following holds for some positive constant $m>0$:
\begin{align}
4C_k r_{k+1}^3 - \frac{\lambda_{k+1}}{c^2} r_{k+1}^2 \leq \frac{2\lambda_{k+1}}{2\lambda_{k+1} + \mu} \frac{4\tilde{\epsilon}_{k+1}}{m} (F(x_k) - F^\star) + 4 \tilde{\epsilon}_{k+1} \left( 1+ \frac{1}{m} \right) F^\star \label{eq:sq_conv_condition1}
\end{align}
We provide the condition on $\lambda$ in the next lemma:

Now if the condition given in equation~\eqref{eq:sq_conv_condition1} holds then the following recursion holds:
\begin{align}
F(x_{k+1})  - F^\star \leq \left( 1+ 4\tilde \epsilon_{k+1}\left( 1+ \frac{1}{m} \right) \right) \frac{2\lambda_{k+1}}{2\lambda_{k+1} + \mu} (F(x_{k})  - F^\star) + 4\tilde \epsilon_{k+1}\left( 1+ \frac{1}{m} \right) F^\star \label{eq:sq_conv_recursion}
\end{align}
We can compare the recursion equations given in equations~\eqref{eq:sq_conv_recursion} and \eqref{eq:recursion_1}. If we choose $\epsilon_{k+1} = 2\tilde{\epsilon}_{k+1}\left( 1  + \frac{1}{m} \right)$, then we have:

\begin{align}
F(x_{k+1})  - F^\star \leq \left( 1+ 2 \epsilon_{k+1}  \right) \frac{2\lambda_{k+1}}{2\lambda_{k+1} + \mu} (F(x_{k})  - F^\star) + 2 \epsilon_{k+1} F^\star \label{eq:sq_conv_recursion}
\end{align}
which also confirms coreset conditions for the original function $F$.

\end{proof}

\begin{lemma}
For a given set of constants $C_k^{(i)} \geq |B_{x_k}^{(i)}(x)|, ~x \in B(x_k,r_{k+1})$ such that $C_k = \frac{1}{n}\sum_{i =1}^n C_k^{(i)}$,  and $\epsilon_k =\delta_k \frac{\mu}{2\lambda_k + \mu}$ for $\delta_k \in (0,1/2)$ $\text{ and } ~\forall~k\in [T]$,  we have ,
$$4C_k r_{k+1}^3 - \frac{\lambda}{c^2} r_{k+1}^2 \leq \frac{2\lambda_{k+1}}{2\lambda_{k+1} + \mu} \frac{4\tilde{\epsilon}_{k+1}}{m} (F(x_k) - F^\star) + 4 \tilde{\epsilon}_{k+1} \left( 1+ \frac{1}{m} \right) F^\star$$
is satisfied if for positive constants $c$ and $m$:
\begin{align*}
 2\lambda_{k+1} = \max\left( \sqrt{ \frac{ 4C_k\| \nabla F(x_k)\|^3}{\frac{1}{4c^2}  {\| \nabla F(x_k) \|^2} +  4 \tilde \delta_{k+1} \mu   \frac{F(x_k)}{3m} }   }- \mu , \mu \right).
\end{align*}
\end{lemma}
\begin{proof}
We need to ensude the following condition:
\begin{align}
4C_k r_{k+1}^3 - \frac{\lambda}{c^2} r_{k+1}^2 \leq \frac{2\lambda_{k+1}}{2\lambda_{k+1} + \mu} \frac{4\tilde{\epsilon}_{k+1}}{m} (F(x_k) - F^\star) + 4 \tilde{\epsilon}_{k+1} \left( 1+ \frac{1}{m} \right) F^\star \label{eq:quad_cond_2}
\end{align}
Let us assume that there exist a positive real number $\theta_{k+1}$.
\begin{itemize}
\item Consider the case when $F(x_k) \geq \theta_{k+1} F^\star$. Hence to ensure the condition given in equation~\eqref{eq:quad_cond_2}, we can just ensure that the following holds:
\begin{align}
4C_k r_{k+1}^3 - \frac{\lambda_{k+1}}{c^2} r_{k+1}^2 \leq  \frac{2\lambda_{k+1}}{2\lambda_{k+1} + \mu} \frac{4\tilde{\epsilon}_{k+1}}{m} \left( 1 - \frac{1}{\theta_{k+1}} \right)F(x_k)   \label{eq:sq_condition_1}
\end{align} 

\item Consider the case when $F(x_k) \leq \theta_{k+1} F^\star$. Then, to ensure the condition given in equation~\eqref{eq:quad_cond_2}, we can just ensure that the following holds:
\begin{align}
4C_k r_{k+1}^3 - \frac{\lambda_{k+1}}{c^2} r_{k+1}^2 \leq  4 \tilde{\epsilon}_{k+1} \left( 1+ \frac{1}{m} \right) \frac{F(x_k)}{\theta_{k+1}}  \label{eq:sq_condition_2}
\end{align} 
\end{itemize}
In equations~\eqref{eq:sq_condition_1} and \eqref{eq:sq_condition_2}, we use $\theta_{k+1} = 1 + (m+1) \frac{2\lambda_{k+1}+ \mu}{ 2\lambda_{k+1}}$ then we get the following condition to be satisfied:
\begin{align}
\begin{split} \label{eq:sq_condition_last}
&4C_k r_{k+1}^3 - \frac{\lambda_{k+1}}{c^2} r_{k+1}^2 \leq 4 \tilde{\epsilon}_{k+1}  \left(1 + \frac{1}{m}\right) \frac{F(x_k)}{\theta_{k+1}} \\
\Rightarrow ~ &4C_k r_{k+1}^3 \leq \frac{\lambda_{k+1}}{c^2} r_{k+1}^2 + 4 \tilde{\epsilon}_{k+1}  \left(1 + \frac{1}{m}\right) \frac{F(x_k)}{\theta_{k+1}}  \\
\Rightarrow ~ & 4C_k \frac{\| \nabla F(x_k)\|^3}{ (2\lambda_{k+1} + \mu)^3}  \leq \frac{1}{2c^2} \frac{2\lambda_{k+1}}{2\lambda_{k+1} + \mu} \frac{\| \nabla F(x_k) \|^2}{2\lambda_{k+1} + \mu } + 4 \tilde{\epsilon}_{k+1}  \left(1 + \frac{1}{m}\right) \frac{F(x_k)}{\theta_{k+1}} 
\end{split}
\end{align}
Now we assume that $2\lambda_{k} \geq \mu ~\forall~ k \Rightarrow \frac{2\lambda_{k}}{2\lambda_k +\mu} \geq \frac{1}{2}$ and $\tilde{\epsilon}_{k+1} = \tilde \delta_{k+1} \frac{\mu}{2\lambda_{k+1}+ \mu} $. Hence the condition given in the equation~\eqref{eq:sq_condition_last} is satisfied when:
\begin{align*}
4C_k \frac{\| \nabla F(x_k)\|^3}{ (2\lambda_{k+1} + \mu)^3} \leq  \frac{1}{4c^2}  \frac{\| \nabla F(x_k) \|^2}{2\lambda_{k+1} + \mu } + 4 \tilde \delta_{k+1} \frac{\mu}{2\lambda_{k+1}+ \mu }   \left(1 + \frac{1}{m}\right) \frac{F(x_k)}{\theta_{k+1}}  \\
\Rightarrow ~ 2\lambda_{k+1} + \mu \geq \sqrt{ \frac{ 4C_k\| \nabla F(x_k)\|^3}{\frac{1}{4c^2}  {\| \nabla F(x_k) \|^2} +  4 \tilde \delta_{k+1} \mu  \left(1 + \frac{1}{m}\right) \frac{F(x_k)}{\theta_{k+1}} }   }
\end{align*}
Now in the above equation we put the value of $\theta_{k+1} =  1  + (m+1) \frac{2\lambda_{k+1}+\mu}{2\lambda_{k+1}} \leq 2m+ 3$. We also use the fact that $m+1 \geq \frac{1}{3} (2m +3)$. That means the other conditions on $\lambda_{k+1}$ are satisfied when 
\begin{align*}
2\lambda_{k+1} + \mu \geq \sqrt{ \frac{ 4C_k\| \nabla F(x_k)\|^3}{\frac{1}{4c^2}  {\| \nabla F(x_k) \|^2} +  4 \tilde \delta_{k+1} \mu   \frac{F(x_k)}{3m} }   }
\end{align*}
Hence, given $$ 2\lambda_{k+1} = \max\left( \sqrt{ \frac{ 4C_k\| \nabla F(x_k)\|^3}{\frac{1}{4c^2}  {\| \nabla F(x_k) \|^2} +  4 \tilde \delta_{k+1} \mu   \frac{F(x_k)}{3m} }   }- \mu , \mu \right),$$
the conditions mentioned in the lemma are satisfied. 
\end{proof}

 \clearpage

\end{document}